\documentclass[10pt]{article}

\title{Optimal Routing across Constant Function Market Makers with Gas Fees
}
\author{C. Escudero\thanks{Departamento de Matem\'aticas Fundamentales,
Universidad Nacional de Educaci\'on a Distancia, Madrid, Spain.
E-mail: cescudero@mat.uned.es, ORCID-ID: 0000-0001-9144-9097.}
\and
F. Lara\thanks{Instituto de Alta Investigaci\'on (IAI), Universidad
de Tarapac\'a, Arica, Chile. E-mail: felipelaraobreque@gmail.com;
flarao@academicos.uta.cl. Web: www.felipelara.cl, ORCID-ID: 0000-0002-9965-0921.}
\and
M. Sama\thanks{Departamento de Matem\'atica Aplicada I,
Universidad Nacional de Educaci\'on a Distancia, Madrid, Spain.
E-mail: msama@ind.uned.es.}
}

\date{\small \emph{\today}}
\usepackage{graphicx}
\usepackage{amsmath}
\usepackage{subcaption}
\usepackage{amsthm}
\usepackage{amssymb}
\usepackage{multicol}
\usepackage{array}
\usepackage[dvipsnames,svgnames]{xcolor}

\numberwithin{equation}{section}

\newcommand{\BE}{\begin{equation}}
\newcommand{\EE}{\end{equation}}
\newcommand{\BEQ}{\begin{eqnarray}}
\newcommand{\EEQ}{\end{eqnarray}}
\newcommand{\BD}{\begin{displaystyle}}
\newcommand{\ED}{\end{displaystyle}}
\newcommand{\BC}{\begin{center}}
\newcommand{\EC}{\end{center}}
\newtheorem{thr}{Theorem}[section]
\newtheorem{defn}{Definition}[section]
\newtheorem{pr}{Proposition}[section]

\newtheorem{ex}{Example}[section]

\def\cone{\operatornamewithlimits{cone}}

\def\cone{\operatornamewithlimits{cone}}

\newtheorem{remark}{Remark}[section]

\begin{document}

\maketitle

\begin{abstract}

\noindent {\bf Abstract.}
We study the optimal routing problem in decentralized exchanges built on Constant Function Market Makers when trades can be split across multiple heterogeneous pools and execution incurs fixed on-chain costs (gas fees). While prior routing formulations typically abstract from fixed activation costs, real on-chain execution presents non-negligible gas fees. They also become convex under concavity/convexity assumptions on the invariant functions. We propose a general optimization framework that allows differentiable invariant functions beyond global convexity and incorporates fixed gas fees through a mixed-integer model that induces activation thresholds.  Subsequently, we introduce a relaxed formulation of this model, whereby we deduce necessary optimality conditions, obtaining an explicit Karush-Kuhn-Tucker system that links prices, fees, and activation. We further establish sufficient optimality conditions using tools from generalized convexity (pseudoconcavity/pseudoconvexity and quasilinearity), yielding a verifiable optimality characterization without requiring convex trade functions. Finally, we relate the relaxed solution to the original mixed-integer model by providing explicit approximation bounds that quantify the utility gap induced by relaxation. Our results extend the mathematical theory for routing by offering no-trade conditions in fragmented on-chain markets in the presence of gas fees.

\medskip


\end{abstract}

\bigskip
\bigskip

\section{Introduction}

Decentralized Finance (DeFi) has emerged as a new paradigm for the provision of
financial services through permissionless and transparent blockchain-based
protocols \cite{gobet2023defi,TTK}. A core building block of DeFi markets are Automated Market Makers
(AMMs), which replace the traditional order book mechanism by on-chain
liquidity pools governed by deterministic pricing rules. Among the different
classes of AMMs, Constant Function Market Makers (CFMMs) have attracted
particular attention due to their mathematical tractability and economic
relevance. In a CFMM, the state of a liquidity pool is characterized by a vector
of reserves whose evolution is constrained by an invariant function, and asset
prices arise endogenously from the gradient of this invariant.

In mathematical terms, a CFMM maintains a set of reserves $(R^1, \ldots, R^n)$ of $n$ different assets in such a way that it keeps invariant a given function
\begin{equation}\nonumber
    \varphi(R^1, \ldots, R^n) = k,
\end{equation}
henceforth denoted as \emph{trade function}, which is usually assumed to be continuously differentiable and to possess a well-defined convexity. The constant $k>0$ is given by the initialization of the CFMM. Examples include the constant-product formula $\varphi(x,y) = xy$ used in Uniswap~v2 \cite{AZR} (for newer versions of Uniswap see \cite{ASZ,AZS,AZT}) and the constant-sum form $\varphi(x,y) = x + y$, which arises
as a theoretical benchmark for stablecoin exchanges \cite{egorov2019}. This structure fixes how exchanges between a user and the CFMM are performed. In particular, it permits an autonomous relative valuation of assets, giving an implicit pricing rule that is derived from the gradient $\nabla \varphi(R^1, \ldots, R^n)$, which determines marginal exchange rates. Hence the assumption of differentiability. Other examples of CFMMs include Balancer \cite{MM} and Curve \cite{egorov2019}.

The mathematical foundations of CFMMs have been rigorously developed in recent
years. In particular, Angeris, Chitra, and coworkers \cite{AT,angeris2022when,angeris2020analysis}
formalized CFMM trading as a convex optimization problem under suitable
regularity and convexity assumptions on the invariant function, providing
fundamental insights into pricing, arbitrage, and optimal trade execution.
Subsequent works extended this framework to multi-asset pools, routing across
multiple markets, and algorithmic execution strategies
\cite{OptRou-2022,CKK,diamandis2024algo}.

In realistic DeFi ecosystems, as these last three references outline, trading does not occur within a single
liquidity pool. Instead, traders face a fragmented market composed of
heterogeneous CFMMs, each characterized by its own invariant function, reserve
composition, and fee structure. Optimal execution may therefore require splitting a trade across several pools, giving rise to the \emph{optimal routing problem}. These previous formulations are based on convex optimization
techniques by imposing global convexity or concavity assumptions on
the invariant functions \cite{angeris2022constant,angeris2020analysis}. These assumptions restrict the admissible design
space of CFMMs and are not always economically or technologically justified. However, most existing theoretical results rely on convexity hypotheses because they guarantee global optimality \cite{HU-L}. This limits the applicability of the
analysis to a narrower class of invariant functions, but is something that can be circumvented by means of nonconvex optimization methods \cite{ELS-1}.

A second structural feature of decentralized markets that challenges standard models is the presence of fixed execution costs. On-chain trades incur
\emph{gas fees}, which must be paid whenever a transaction interacts with a
liquidity pool, independently of the traded volume. These costs fundamentally alter optimal routing decisions. From a financial perspective, gas fees expand no-trade
regions, suppress arbitrage opportunities, and interact with liquidity
fragmentation in non-trivial ways. Related phenomena have been discussed in the literature, including miner extractable value (MEV) and other inefficiencies in DeFi markets
\cite{capponi2021adoption,capponi2024maximal,capponi2023paradox}. However, despite their importance, fixed execution costs are often either neglected or
handled heuristically \cite{DA}, and their interaction with the
underlying mathematical structure of CFMMs remains insufficiently understood to the best of our knowledge.

The aim of this paper is to address these limitations by developing a general
optimization-theoretic framework for optimal routing across multiple
heterogeneous CFMMs that (i) allows differentiable invariant functions beyond
global convexity assumptions, and (ii) explicitly incorporates gas fees through
market activation variables and capacity constraints.

\subsection*{Our contribution}

Our contributions in this work can be summarized as follows.

\begin{itemize}
\item We formulate the optimal routing problem across multiple CFMMs in the presence of fixed gas fees and without convexity assumptions. This is done at the price of introducing market activation variables and
bounded trade constraints. The resulting model captures some essential features of on-chain execution while remaining amenable to rigorous mathematical analysis through the introduction of a relaxation approximation.

\item For the relaxed routing problem, we derive \emph{necessary optimality
conditions} in the form of a Karush-Kuhn-Tucker (KKT) system under a
Kurcyusz-Robinson-Zowe-type constraint qualification adapted to the structure
of CFMM trading constraints. This characterization explicitly links marginal
utilities, pool prices, fees, and activations.

\item Using tools from generalized convexity, we establish \emph{sufficient
optimality conditions} without assuming global convexity of the invariant functions. In particular, under pseudoconcavity of the utilities and
quasilinearity of the trade functions, the KKT conditions are shown to be both
necessary and sufficient for global optimality.

\item We analyze the relationship between the relaxed routing problem and the
original mixed-integer formulation with fixed activation costs. Explicit
approximation bounds are derived, quantifying the utility loss induced by the
relaxation and providing a rigorous interpretation of relaxed routing solutions.

\end{itemize}

Overall, our results extend the mathematical theory of CFMM routing beyond convex models and provide a mathematical framework for understanding the impact of gas fees on optimal trading and no-trade conditions in fragmented decentralized markets.

\section{Preliminaries}
 
We denote 
$\mathbb{R}_{+} := [0, + \infty[$ and $\mathbb{R}_{++} := \, ]0, + \infty[$, thus $\mathbb{R}^{n}_{+}:= (\mathbb{R}_{+})^{n}$ and $\mathbb{R}^{n}_{++}:= (\mathbb{R}_{++})^{n}$, respectively. We use the usual notations $\geq$ componentwise (see \cite{ehrgott2005multicriteria}), that is,
\begin{align*}
 & ~~ x \geq y ~ \Longleftrightarrow ~ x_{i} \geq y_{i}, \\
 & ~~ x \gneq y ~ \Longleftrightarrow ~ x_{i} \geq y_{i}, \, x \neq y, \\
 & ~~ x > y ~ \Longleftrightarrow ~ x_{i} > y_{i},
\end{align*}
for every $i=1,\ldots,n$.


Given any extended-valued function $h: \mathbb{R}^{n} \rightarrow
\overline{\mathbb{R}} := \mathbb{R} \cup \{\pm \infty\}$, 
it is indicated by ${\rm epi}\,h := \{(x,t) \in \mathbb{R}^{n} \times \mathbb{R}: h(x) \leq t\}$ the epigraph of $h$, by $S_{\lambda} (h) :=
\{x \in \mathbb{R}^{n}: h(x) \leq \lambda\}$ 
the sublevel 
set of $h$ at the height $\lambda \in \mathbb{R}$, by $W_{\lambda} (h) := \{x \in \mathbb{R}^{n}: h(x) \geq \lambda\}$ 
the upper level 
set of $h$ at the height $\lambda \in \mathbb{R}$, and by ${\rm argmin}_{\mathbb{R}^{n}} h$ the set 
of all minimal points of $h$ on $\mathbb{R}^{n}$.

A function $h$ with convex domain is said to be
\begin{itemize}
 \item[$(a)$] convex if, given any $x, y \in \mathrm{dom}\,h$, then
 \begin{equation}\label{def:convex}
  h(\lambda x + (1-\lambda)y) \leq \lambda h(x) + (1 - \lambda) h(y), ~ \forall ~ \lambda \in [0, 1], \notag
 \end{equation}

 \item[$(b)$] quasiconvex if, given any $x, y \in \mathrm{dom}\,h$, then
 \begin{equation}\label{def:qcx}
  h(\lambda x + (1-\lambda) y) \leq \max \{h(x), h(y)\}, ~ \forall ~ \lambda \in [0, 1], \notag
 \end{equation}
\end{itemize}
 Every convex function is quasiconvex, but the reverse statement does not hold as the function $h(x) = x^{3}$ shows. Recall that
 \begin{align*}
  h ~ \mathrm{is ~ convex} & \Longleftrightarrow \, \mathrm{epi}\,h ~\mathrm{is ~ a ~ convex ~ set;}\\ 
  h ~ \mathrm{is ~ quasiconvex} & \Longleftrightarrow \, S_{\lambda} (h) ~ \mathrm{is ~ a ~ convex ~ set ~ for ~ all ~ } \lambda \in \mathbb{R}.
 \end{align*}

Quasiconvex functions appear in many applications from different fields as, for instance, in Economics and Financial Theory, especially in consumer pre\-fe\-ren\-ce theory (see \cite{D-1959,MWG}), since quasiconcavity is the mathematical formulation of the natural assumption of a {\it tendency to diversification} on the consumers. 

It is said that $h$ is quasilinear if $h$ is quasiconvex and $-h$ is quasiconvex.
As a consequence,  its sublevel set $S_{\lambda} (h)$ and its   upper level
sets $W_{\lambda} (h)$  are convex for all $\lambda \in \mathbb{R}$ (see
\cite[Theorem 3.3.1]{CMA}).

Let $K \subseteq \mathbb{R}^{n}$ be a convex set and $h: K \rightarrow \mathbb{R}$ be a differentiable function. Then, the following assertions hold:
\begin{itemize}
 \item[$(i)$] $h$ is quasiconvex if and only if for every $x, y \in K$, we have (see \cite{AE} and \cite[Theorem 3.11]{ADSZ}) that
\begin{equation}\label{char:AE}
 h(x) \leq h(y) ~ \Longrightarrow ~ \langle \nabla h(y), x - y \rangle \leq 0;
\end{equation}

 \item[$(ii)$] $h$ is quasilinear if and only if for every $x, y \in K$, we have (see \cite[Theorem 3.3.6]{CMA})
 \begin{equation}\label{char:quasilinear}
  h(x) = h(y) ~ \Longrightarrow ~ \langle \nabla h(y), x - y \rangle = 0.
 \end{equation}
 \end{itemize}

 Let $h: \mathbb{R}^{n} \rightarrow \mathbb{R}$ be a differentiable function. Then $h$ is said to be pseudoconvex (see \cite{Manga}) if
\begin{equation}\label{pseudoconvex}
 \langle \nabla h(x), y - x \rangle \geq 0 ~ \Longrightarrow ~ h(y) \geq h(x).
\end{equation}
A function $h$ is pseudoconcave if $-h$ is pseudoconvex. Furthermore, if $h$ is pseudoconvex, then every local minimum is global minimum \cite[Theorem 3.2.5]{CMA}, a property that quasiconvex functions do not have.

For a further study on generalized convexity, we refer to \cite{AE,ADSZ,CMA,HKS,Manga,SZ} among others.

\section{Optimal Routing with Gas Fee}\label{sec:03}

Our basic framework is that of \cite{OptRou-2022}, which we summarize in the following. We consider a system of $m\in \mathbb{N}$, $m\geq 1$, number of CFMMs, which trades multiple tokens (or crypto-assets) from an universe of $n \in \mathbb{N}$, $n \geq 2$, tokens. Let $n$ be the specific number of tokens that we could trade. Each market $i$ has $n_{i} \in \mathbb{N}$ tokens with $n_{i} \in \{2, \ldots, n\}$ for each $i \in \{1, \ldots, m\}$, and each market $i$ is associated with a connectivity matrix $A^{i} \in \mathbb{R}^{n \times n_{i}}$, with $i \in \{1,\ldots ,m\}$, which relates the global index of a token to its local index. Specifically, $A_{jk}^{i} = 1$ when the global token $j$ corresponds to the local index $k \in \{1, \ldots, n_{i}\}$ in each CFMM $i$.

Each CFMM $i$ has a quantity of reserves $R^{i} \in \mathbb{R}_{+}^{n_{i}}$ and requires a relative tax payment of $\gamma_{i} \in \, ]0,1[$ for each trade. We recall that the CFMM $i$ accepts a proposed trade $(x^{i},y^{i})\in \mathbb{R}_{+}^{n_{i}} \times \mathbb{R}_{+}^{n_{i}}$ when its trade function $\varphi_{i}: \mathbb{R}_{+}^{n_{i}} \rightarrow \mathbb{R}$ satisfies that 
\begin{equation}\label{trade:function}
\varphi_{i} (R^{i}+\gamma _{i}y^{i}-x^{i}) = \varphi_{i}(R^{i}).
\end{equation}

Based on \cite[Section 5]{OptRou-2022}, we are interested in finding a set of valid trades that maximizes the utility of the trader $u: \mathbb{R}^{n} \rightarrow \mathbb{R}$ in the presence of gas fees $q \in \mathbb{R}_{+}^{m}$, where each $q_{i}\in \mathbb{R}_{+}$ corresponds to CFMM $i$. The gas fee is the cost of recording a trade in the blockchain network, which will be assumed to be fixed for simplicity. We also introduce $b^{i} \in \mathbb{R}_{++}^{n_{i}}$, where $b_{k}^{i}$ is the maximum amount of asset $k$ that the trader can buy from CFMM $i$. Hence, the optimal routing problem with gas fee is given by (see \cite[Section 5]{OptRou-2022})
\begin{equation}
\begin{aligned}
&\text{maximize} && u(\Psi ) - \langle q, \eta \rangle \\ 
&\text{subject to} && \Psi = \sum_{i=1}^{m} A^{i} (x^{i} - y^{i}), \\ 
&&& \varphi_{i}(R^{i} + \gamma_{i} y^{i} - x^{i}) = \varphi_{i}(R^{i}), \\ 
&&& 0 \leq x^{i} \leq R^{i},  \\ 
&&& 0 \leq y^{i} \leq \eta_{i} b^{i}, \\ 
&&& \eta_{i} \in \{0,1\}, ~ \forall ~ i \in \{1, \ldots, m\}.
\end{aligned}
\tag{$\mathcal{P}(\mathbf{b},q)$}
\label{GF}
\end{equation}

In this model, $\eta_{i} \in \{0,1\}$ serves as a binary activation indicator for CFMM $i$, taking the value $\eta_{i} = 1$ when a trade is executed in pool $i$, and $\eta_{i} = 0$ other\-wi\-se. The inequality constraint $y^{i} \leq \eta_{i}b^{i}$ smoothly enforces market activation by bounding trade amounts $y^{i}$ by $b^{i}$ when active ($\eta_i = 1$) while re\-qui\-ring $y^{i} = 0$ when inactive. Our notation emphasizes the crucial parameters $\mathbf{b} := (b^{1},\ldots,b^{m})$ (activation bounds) and $q$ (vector of gas fees), which are fundamental to the analysis.

Furthermore, we assume that $b^{i} \in \mathbb{R}_{++}^{n_{i}}$ for all $i \in \{1,\ldots,m\}$, i.e., 
\begin{equation}\label{assump:b}
 b_{k}^{i} \neq 0, ~~ \forall ~ i \in \{1, \ldots, m\}, ~ \forall ~ k \in \{1, \ldots, n_{i}\}.
\end{equation}
The assumption \eqref{assump:b} means that each CFMM $i \in \{1, \ldots, m\}$ allows for a positive maximal injection $b^{i}_{k}$ of the cryptoasset $k \in \{1, \ldots, n_{i}\}$.

Following \cite{OptRou-2022}, we relax the parameters $0 \leq \eta _{i}\leq 1$, in order to obtain a more computable form of the problem, that is, 
\begin{equation}
\begin{array}{ll}
\text{maximize} & u(\Psi )-\langle q,\eta \rangle  \\ 
\text{subject to} & \Psi =\displaystyle{\sum_{i=1}^{m}A^{i}}{(x^{i}-y^{i})},
\\ 
& \varphi_{i} (R^{i}+\gamma _{i}{y^{i}}-{x^{i}})=\varphi _{i}(R^{i}), \\ 
& 0\leq x^{i} \leq R^{i}, \\ 
& 0 \leq y^{i}\leq \eta _{i}b^{i}, \\ 
& 0\leq \eta_{i} \leq 1, ~ \forall ~ i \in \{1, \ldots ,m\}.\ \tag{$\mathcal{P}^{r}(\mathbf{b},q)$}%
\end{array}
\label{rGF}
\end{equation}%
We use the notation 
\begin{equation*}
 (\mathbf{x}, \mathbf{y}, \mathbf{\eta}) = (x^{1}, \ldots,x^{m},y^{1}, \ldots, y^{m},\eta_{1}, \ldots,\eta_{m}) \in \prod\limits_{i=1}^{m} \mathbb{R}_{+}^{n_{i}} \times \prod\limits_{i=1}^{m} \mathbb{R}_{+}^{n_{i}} \times [0,1]^{m}, 
\end{equation*}
where $[0,1]^{m} := \prod\limits_{i=1}^{m} [0,1]$ while $\mathbf{x} := (x^{1}, \ldots, x^{m})$, $\mathbf{y} := (y^{1}, \ldots, y^{m})$ and $\eta := (\eta_{1}, \ldots, \eta_{m})$ denote, respectively, a potential trade and the vector of market activation relaxed pa\-ra\-me\-ters. We stress that the relaxed activation variables do not represent executable
strategies \emph{per se}, but serve as analytical proxies to characterize optimality
and approximation properties.

In the same way, by $A := (A^{1},\ldots ,A^{m}): \prod\limits_{i=1}^{m} \mathbb{R}^{n_{i}} \rightarrow \mathbb{R}^{n}$ we denote the linear operator defined by 
\begin{equation*}
 A(\mathbf{x}) = {\sum_{i=1}^{m} A^{i}x^{i}},
\end{equation*}
Then, problem \eqref{rGF} can be equivalently written as: 
\begin{equation*}
\begin{array}{ll}
\text{maximize} & (u\circ A) (\mathbf{x}-\mathbf{y}) - \langle q, \eta \rangle 
\\ 
\text{subject to} & \varphi_{i} (R^{i}+ \gamma_{i}{y^{i}} - {x^{i}}) - \varphi_{i} (R^{i})=0, \\ 
& 0\leq x^{i}\leq R^{i}, \\
& y^{i}-\eta _{i}b^{i}\leq 0,\  \\ 
& 0 \leq y^{i}, \\
& 0 \leq \eta_{i} \leq 1, ~ \forall ~i\in \{1,\ldots,m\}, 
\end{array}
\end{equation*}
which we can model as an equivalent minimization problem with equa\-li\-ty and inequality constraints: 
\begin{equation}
\begin{array}{ll}
\text{minimize} & f(\mathbf{x},\mathbf{y},\eta) := (-u \circ A)(\mathbf{x} - \mathbf{y}) + \langle q, \eta \rangle \\ 
\text{subject to} & G_i(\mathbf{x},\mathbf{y},\eta) := y^i - \eta_i b^i \leq 0, \tag{$\mathcal{P}^{r}(\mathbf{b},q)$} \\ 
& H_i(\mathbf{x},\mathbf{y},\eta) := \varphi_i(R^i + \gamma_i y^i - x^i) - \varphi_i(R^i) = 0, \\ 
& \forall ~ i = 1,\dots,m, \\ 
& (\mathbf{x},\mathbf{y},\eta) \in \widehat{\mathbf{S}},
\end{array}
\label{rP}
\end{equation}
where $f: \prod\limits_{i=1}^{m} \mathbb{R}_{+}^{n_{i}} \times \prod\limits_{i=1}^{m} \mathbb{R}_{+}^{n_{i}} \times [0,1]^{m} \rightarrow \mathbb{R}$, $\mathbf{\ } G_{i}: \prod\limits_{i=1}^{m}
\mathbb{R}_{+}^{n_{i}} \times \prod\limits_{i=1}^{m} \mathbb{R}_{+}^{n_{i}} \times [0,1]^{m} \rightarrow \mathcal{Z}$, $H_{i}: \prod\limits_{i=1}^{m} \mathbb{R}_{+}^{n_{i}} \times
\prod\limits_{i=1}^{m} \mathbb{R}_{+}^{n_{i}} \times [0,1]^{m} \rightarrow \mathbb{R}^{m}$ and the feasible set is
\begin{equation*}
 \mathbf{\hat{S}} := \left\{ \mathbf{(x,y,}\eta \mathbf{)}: \, 0 \leq x^{i}\leq R^{i}, \text{ }0 \leq {y^{i}},\text{ }0 \leq \eta _{i}\leq 1, ~ \forall ~ i \in \{1,\ldots ,m\}\right\} .
\end{equation*}
 

\subsection{Necessary Optimality Conditions}

Under the stated conditions, all problems can be solved by applying standard results (see, for example, \cite[Theorem 2.3]{jahn2020introduction}). 

Throughout this paper, we assume that the utility function $u$ and the market functions $\varphi _{i}$ are differentiable for every $i = 1, \ldots, m$, that $({\bf \bar{x}}, {\bf \bar{y}}, \bar{\eta}) \mathbf{\in \hat{S}}$ is a solution to problem \eqref{rP} and that $\mathbf{(x,y,} \eta \mathbf{)\in } \prod\limits_{i=1}^{m} \mathbb{R}^{n_{i}} \times \prod\limits_{i=1}^{m} \mathbb{R}^{n_{i}} \times [0,1]^{m}$ is an arbitrary direction, respectively. We also consider the following assumptions:
\begin{itemize}
 \item Positiveness of reserve and maximum of tendered baskets, 
\begin{equation}\label{p:Rb}
 \mathbf{R,b\in }  \prod\limits_{i=1}^{m} \mathbb{R}_{++}^{n_{i}}.
\end{equation}

 \item Non-overlapping support (see \cite[Section 3.2]{angeris2022constant}), that is, the following complementary condition is verified for every trade
\begin{equation}\label{def:comple}
 0 \leq \bar{x}^{i} \bot \bar{y}^{i} \geq 0 ~ \Longleftrightarrow ~ \bar{x}^{i} \geq 0, \, \bar{y}^{i} \geq 0\text{, } \bar{x}_{j}^{i} \, \bar{y}_{j}^{i} = 0, ~ \forall ~ i,j,
\end{equation}
where $i \in \{1, \ldots, m\}$ and $j \in \{1, \ldots, n_{i}\}$.

\item Increasing behavior of trading and utility functions (see \cite{ELS-1}):
\begin{align} 
 & ~ \nabla u(x) \geq 0 ~ \text{ for every }  x\in \mathbb{R}^{n}, \label{p:utility} \\
 & \nabla \varphi_{i}(x) \gneq 0\text{ }\ \text{for every }x>0\text{ and every }i \in \{1,\ldots ,m\}, \label{p:market}
\end{align}
where the last one implies that $\varphi_{i}$ is strictly increasing with respect to at least one component and, in particular, 
 
\begin{equation*} 
x\geq y>0~\Longrightarrow ~\varphi _{i}(x)\geq \varphi _{i}(y)\ \text{for
every }i=1,\ldots ,m.
\end{equation*}  
\end{itemize}

We define the price vector (see \cite{angeris2022constant}) and the updated price vector by
\begin{align*}
 & P^{i} := \nabla \varphi _{i}(R^{i}),\ \text{ for every } i = 1, \ldots,m, \\
 & \bar{P}^{i} := \nabla \varphi _{i}(R^{i} + \gamma_{i} {\bar{y}^{i}}-{\bar{x}^{i}}), \text{ for every } i = 1,\ldots,m.
\end{align*}

Under condition \eqref{p:market}, $P^{i}, \bar{P}^{i} \neq 0$ for every $i=1,\ldots,m$. 

\begin{remark} \label{rem:mu}
Under the stated conditions, problem \eqref{rP} is solvable and has at least one solution $\left( \mathbf{\bar{x}}, \mathbf{\bar{y}}, \bar{\eta}\right)$ by applying standard results (see, e.g., \cite[Theorem 2.3 ]{jahn2020introduction}). 
Furthermore, if $\bar{\eta}_{i}=0$, then $\bar{x}^{i} = \bar{y}^{i}=0$.
 
Conversely, if $\bar{x}^{i}=\bar{y}^{i}=0$ for every $i\in \{1,\dots,m\}$, i.e., $\mathbf{\bar{x}} = \mathbf{\bar{y}} = \mathbf{0}$, then $\mathbf{\bar{\eta}}=0$. Otherwise, if some $0 < \bar{\eta}_{r} \leq 1$, we have 
\begin{equation*}
f(\mathbf{0},\mathbf{0},0) = -u \left( A(\mathbf{0}) \right) < -u \left( A(\mathbf{0}) \right) + q_{r} \bar{\eta}_{r} \leq f(\mathbf{0}, \mathbf{0}, \bar{\eta}),
\end{equation*}
which violates optimality.
\end{remark}

A necessary condition is given by a standard multiplier rule \cite[Theo\-rem 5.3]{jahn2020introduction}: Assume that $(\mathbf{\bar{x}},\mathbf{\mathbf{\bar{y}},} \bar{\eta})$ solves problem \eqref{rP} and a constraint qualification holds, in particular,  if the Kurcyusz-Robinson-Zowe (KRZ) constraint qua\-li\-fi\-ca\-tion is verified:
\begin{eqnarray} \nonumber
\left( 
\begin{array}{c}
 \nabla G({\bf \bar{x}}, {\bf \bar{y}}, \bar{\eta}) \\ 
 \nabla H({\bf \bar{x}}, {\bf \bar{y}}, \bar{\eta})
\end{array}
\right) \cone(\mathbf{\hat{S}}-({\bf \bar{x}}, {\bf \bar{y}}, \bar{\eta})) + \cone\left( 
\begin{array}{c}
\prod\limits_{i=1}^{m}\mathbb{R}_{+}^{n_{i}} + \left\{ G({\bf \bar{x}}, {\bf \bar{y}},\bar{
\eta}) \right\}  \\ 
0
\end{array}%
\right)
\\ \label{KRZ}
= 
\left( 
\begin{array}{c}
\prod\limits_{i=1}^{m} \mathbb{R}^{n_{i}} \\ 
\mathbb{R}^{m} 
\end{array} 
\right), 
\end{eqnarray}
then the following multiplier rule is verified: Find $\left( \left( \mathbf{\bar{x}}, \mathbf{\mathbf{\bar{y}},} \bar{\eta} \right), \left( \mu^{1},\ldots,\mu^{m} \right) ,\lambda \right) \in \mathbf{\hat{S}} \times \left( \prod\limits_{i=1}^{m} \mathbb{R}_{+}^{n} \right) \times \mathbb{R}^{m}$ such that 
\begin{flalign}
 & \!\! \left( \! \nabla f(\mathbf{\bar{x}, \bar{y},} \bar{\eta} \mathbf{)} \! + \!\! \sum\limits_{i=1}^{m} (\mu^{i})^{\top} \nabla G_{i} (\mathbf{\bar{x}, \bar{y},} \bar{\eta}) \! + \!\! \sum\limits_{i=1}^{m} \lambda _{i} \nabla H_{i}(\mathbf{\bar{x}, \bar{y},} \bar{\eta}) \! \right) \! (\mathbf{x-\bar{x} \mathbf{,y-\bar{y}\mathbf{,}}}\eta \mathbf{-} \bar{\eta}) \geq 0, \notag  \\ 
 & (\mu^{i})^{\top} G_{i} (\mathbf{\bar{x}, \bar{y},} \bar{\eta}) = 0,\ \text{ for every }
i=1,\ldots,m, \label{KKT-1} \\ 
 & H_{i} (\mathbf{\bar{x}}, \mathbf{\bar{y}}, \bar{\eta}) =0 \ \text{for every } i=1,\ldots,m, \text{ and }(\mathbf{x}, \mathbf{\mathbf{y},} \eta ) \in \mathbf{\hat{S}} \text{ arbitrary.} \notag
\end{flalign}

In the following, we note that, by a direct calculation, we have
\begin{align*}
 & \nabla f(\mathbf{\bar{x}, \bar{y},} \bar{\eta}) ({\bf x}, {\bf y},\eta) = \sum_{i=1}^{m} \nabla u \left( A(\mathbf{\bar{x}} - \mathbf{\bar{y}}) \right) A^{i} y^{i} -{\sum_{i=1}^{m}} \nabla u \left( A(\mathbf{\bar{x}} - \mathbf{\bar{y}}) \right) A^{i}x^{i} + q \cdot \eta, \\
 & \nabla G_{i} (\mathbf{\bar{x},\bar{y},} \bar{\eta}) (\mathbf{x,y,} \eta) = y^{i}-\eta _{i} b^{i}, \\
 & \nabla H_{i} (\mathbf{\bar{x}, \bar{y},} \bar{\eta}) (\mathbf{x,y,}\eta) = - \nabla \varphi_{i} (\mathbf{R} + \gamma  \mathbf{\bar{y}} - \mathbf{\bar{x})} x^{i} \mathbf{+} \gamma_{i} \nabla \varphi_{i}(\mathbf{R} + \gamma \mathbf{\bar{y}} - \mathbf{\bar{x})} y^{i},
\end{align*}
for every $i \in \{1,\ldots,m\}$. 

In the next result, we establish conditions for which the constraint qua\-li\-fi\-ca\-tion \eqref{KRZ} is verified.

\begin{pr}\label{lm:KRZ} 
 Assume that properties \eqref{p:Rb} and \eqref{p:market} hold. Let $(\mathbf{\bar{x}}, \mathbf{\bar{y}}, \bar{\eta})$ be a solution to \eqref{rP} satisfying \eqref{def:comple}.   Then, constraint qualification \eqref{KRZ} holds at $(\mathbf{\bar{x}}, \mathbf{\bar{y}}, \bar{\eta}) = (\mathbf{0}, \mathbf{0}, 0)$. Moreover, when $(\mathbf{\bar{x}}, \mathbf{\bar{y}}, \bar{\eta}) \neq (\mathbf{0}, \mathbf{0}, 0)$, condition \eqref{KRZ} is also satisfied if the following additional property holds:
\begin{equation}\label{p:y}
 0 < b_{j}^{i} - \bar{y}_{j}^{i} \leq \frac{2}{\gamma_{i}} R_{j}^{i} ~ \text{with} ~ \bar{y}_{j}^{i} > 0, ~ \forall ~ j \in \{1, \ldots,n_{i}\}, ~ \forall ~ i \in \{1, \ldots, m\}.
\end{equation}
\end{pr}

\begin{proof}
 Condition \eqref{KRZ} is equivalent to: For every $((z^{1},\ldots,z^{n}), t) \in \! \prod\limits_{i=1}^{m} \mathbb{R}^{n_{i}} \times \mathbb{R}$ we can take $\alpha ,\beta \in \mathbb{R}_{+}$, $({\bf x}, {\bf y}, \eta) \in \mathbf{\hat{S}}$ and $c,d \in \prod\limits_{i=1}^{m} \mathbb{R}_{+}^{n_{i}}$ such that  
\begin{equation*}
\begin{array}{l} 
 z^{i} = \nabla G_{i} ({\bf \bar{x}}, {\bf \bar{y}}, \bar{\eta}) (\alpha (x^{i} - \bar{x}^{i}, y^{i} -\bar{y}^{i}, \eta_{i} - \bar{\eta}_{i})) + \beta (c^{i} + \left\{ G_{i} ({\bf \bar{x}}, {\bf \bar{y}}, \bar{\eta}) \right\}), \\ 
 t_{i} = \nabla H_{i} ({\bf \bar{x}}, {\bf  \bar{y}}, \bar{\eta})(\alpha (x^{i} - \bar{x}^{i}, y^{i} -\bar{y}^{i}, \eta_{i} - \bar{\eta}_{i})). 
\end{array}
\end{equation*}  
Indeed, by a direct calculation, the latter is equivalent to: There exist $0 \leq x^{i} \leq R^{i}$, $0 \leq y^{i}$ and $0 \leq \eta_{i} \leq 1$ such that 
\begin{subequations}
\begin{align}
 & z_{j}^{i} = \alpha \left( (y^{i}_{j} - \bar{y}^{i}_{j}) - (\eta_{i} - \bar{\eta}_{i}) b^{i}_{j} \right) + \beta (c^{i}_{j} + \bar{y}^{i}_{j} - \bar{\eta}_{i} b^{i}_{j}),  \label{180224a} \\
 & \, t_{i} = \alpha (\bar{P}^{i})^{\top} \left( - (x^{i} -\bar{x}^{i}) + \gamma_{i} (y^{i} - \bar{y}^{i}) \right), \label{180224b}
\end{align}
for every $i \in \{1, \ldots, m\}$ and every $j \in \{1, \ldots, n_{i}\}$.

Without loss of generality, in the sequel we consider a fixed $i \in \{1, \ldots ,m\}$. By hypothesis \eqref{p:market}, we can assume 
\end{subequations}
\begin{equation*}
 \bar{P}_{k}^{i} \neq 0,
\end{equation*}
for some  $k \in \{1, \ldots, n_{i}\}$. Then, we consider the following cases:

If $\bar{\eta}_{i} = 0$, then $(\bar{x}^{i}, \bar{y}^{i}) = (0,0)$ by Remark \ref{rem:mu}. In this case, equations \eqref{180224a} and \eqref{180224b} reduce to
\begin{subequations}
\begin{align}\label{270325a}
 z^{i} & =\alpha \left( y^{i} - \eta_{i} b^{i} \right) + \beta c^{i}, \\
 t_{i} & = \alpha (\bar{P}^{i})^{\top} \left( -x^{i} + \gamma_{i} y^{i}\right). \label{270325b}
\end{align}
Setting $\alpha = \beta$, equation \eqref{270325a} simplifies to 
\end{subequations}
\begin{equation}\label{270325c}
 z^{i} = \alpha \left( y^{i} + c^{i} - \eta_{i} b^{i} \right).
\end{equation}

We may assume $\eta_{i} b^{i} - y^{i} \in \mathrm{int} (\mathbb{R}%
_{+}^{n_{i}})$, for instance choosing $y^{i} = \frac{1}{2} \eta_{i} b^{i}
\in \mathrm{int} (\mathbb{R}_{+}^{n_{i}})$ with $\eta_{i} > 0$. On the other
hand, we can define: 
\begin{equation}
c^{i} = \eta_{i} b^{i}- y^{i} + \frac{1}{\alpha} z^{i} = \frac{1}{2}
\eta_{i} b^{i} + \frac{1}{\alpha} z^{i} \geq 0,  \label{280325}
\end{equation}
which holds for sufficiently large $\alpha$, thus there exists $\alpha_{1} > 0$ such
that \eqref{280325} for every $\alpha \geq \alpha_{1}$.

 Let $x^{i}=\gamma _{i}y^{i}-\dfrac{1}{\alpha }\frac{t_{i}}{\bar{P}_{k}^{i}%
}e^{k}=\frac{1}{2}\gamma _{i}\eta _{i}b^{i}-\dfrac{1}{\alpha }\frac{t_{i}}{%
\bar{P}_{k}^{i}}e^{k}$, where $e^{k}$ (with $e_{j}^{k}=\delta _{jk}$) is the 
$k$-th standard basis vector. Without loss of generality we consider $y^{i} =  \frac{1}{2} \eta _{i}b^{i}\leq R^{i}$. Thus, taking $\eta _{i}$ small enough, we have 
\begin{equation}
x^{i}=\gamma _{i}y^{i}-\frac{1}{\alpha }\frac{t_{i}}{\bar{P}_{k}^{i}}%
e^{k}\leq \frac{1}{2}\gamma _{i}\eta _{i}b^{i}\leq R^{i}.  \notag
\end{equation}
At the same time there exists $\alpha _{2}>0$ large enough such that 
\begin{equation}
x^{i}=\gamma _{i}y^{i}-\frac{1}{\alpha }\frac{t_{i}}{\bar{P}_{k}^{i}}%
e^{k}\geq 0,\mathrm{~for~all~}\alpha \geq \alpha _{2}.  \label{2803252}
\end{equation}
Taking $\alpha := \max\{\alpha_{1}, \alpha_{2}\} > 0$ and substituting \eqref{280325} and \eqref{2803252} into \eqref{270325c} and  \eqref{270325b}, we obtain: 
\begin{align}
 \alpha (y^{i} + (\eta_{i} b^{i} + \frac{z^{i}}{\alpha} - y^{i}) - \eta_{i} b^{i}) & = z^{i}, \label{eq:z} \\
 \alpha (\bar{P}^{i})^{\top} (-( \gamma_{i} y^{i} - \frac{1}{\alpha} \frac{t_{i}}{\bar{P}_{k}^{i}} e^{k}) + \gamma_{i} y^{i}) & = t_{i}. \label{eq:t}
\end{align}%
Thus, equations \eqref{270325a} and \eqref{270325b} are satisfied, that is, the constraint qualification condition holds at $(\bar{\mathbf{x}}, \bar{\mathbf{y}}, \bar{\eta}) = (0,0,0)$.

Let us analyze the case when $\bar{\eta}_{i} \neq 0$. If $\bar{\eta}_{i} \neq 0$, then by a similar reasoning as before we consider $\alpha = \beta$, thus replacing in \eqref{180224a} and \eqref{180224b}, we have
\begin{subequations}
\begin{align}
 & z^{i} = \alpha ( y^{i} - \eta_{i} b^{i} + c^{i}), \\
 & t_{i} = \alpha (\bar{P}^{i})^{\top} (-(x^{i} - \bar{x}^{i}) + \gamma_{i} (y^{i} - \bar{y}^{i})).
\end{align}
\end{subequations}
In this case, if $y^{i} = \frac{1}{2} (\bar{y}^{i} + b^{i})$ and $\eta _{i}=1$, then $y^{i} - \bar{y}^{i} = \dfrac{1}{2} \left( b^{i}-\bar{y}^{i} \right) \in {\rm int} (\mathbb{R}_{+}^{n_{i}})$ by \eqref{p:y}, and we can assure that $c^{i} := \dfrac{1}{2} \left( b^{i} - \bar{y}^{i} \right) + \frac{z^{i}}{\alpha} \in \mathbb{R}_{+}^{n_{i}}$ and
\begin{equation*}
 x^{i} := \bar{x}^{i} + \gamma_{i} (y^{i} - \bar{y}^{i}) - \frac{1}{\alpha} \frac{t_{i}}{\bar{P}_{k}^{i}} e^{k} = \bar{x}^{i} + \frac{\gamma_{i}}{2} (b^{i} - \bar{y}^{i}) - \frac{1}{\alpha} \frac{t_{i}}{\bar{P}_{k}^{i}} e^{k} \in \mathbb{R}_{+}^{n_{i}},
\end{equation*}
for $\alpha>0$   big  enough. Furthermore, we can assure that $x^{i} \leq R^{i}$ if 
\begin{equation}\label{by:compcond}
 \bar{x}^{i} + \dfrac{\gamma_{i}}{2} (b^{i} - \bar{y}^{i}) \leq R^{i} ~ \Longleftrightarrow ~ 2 \bar{x}^{i} + \gamma_{i} (b^{i} - \bar{y}^{i}) \leq 2R^{i}.
\end{equation}
By the complementary condition \eqref{def:comple}, relation \eqref{by:compcond} is equivalent to $2 \bar{x}_{j}^{i} \leq 2 R_{j}^{i}$ for every $j$ such that $\bar{x}_{j}^{i} > 0$, which is evident, and $\gamma_{i} (b_{j}^{i} - \bar{y}_{j}^{i}) \leq 2 R_{j}^{i}$ for $\bar{y}_{j}^{i}>0$, which holds in virtue of assumption \eqref{p:y}.
\end{proof}

In the following result, we give an explicit form for the KKT system \eqref{KKT}. Furthermore, in order to avoid confusion with variables, we emphasize that we denote partial derivatives by 
$$D_{k} u\equiv \partial u/\partial x_{k}.$$

\begin{thr} \label{nec:KKT}
 Assume that properties \eqref{p:Rb}, \eqref{p:utility}, and \eqref{p:market} hold. Also, let $(\mathbf{\bar{x}}, \mathbf{\bar{y}}, \bar{\eta}) \in \mathbf{\hat{S}}$ be a solution to \eqref{rP} satisfying \eqref{def:comple} and \eqref{p:y} whenever $(\mathbf{\bar{x}}, \mathbf{\bar{y}}, \bar{\eta}) \neq (\mathbf{0}, \mathbf{0}, 0)$. Then, there exist $\mu = (\mu^1, \ldots, \mu^m) \in \prod_{i=1}^m \mathbb{R}_+^{n_i}$ and $\alpha \in \mathbb{R}_+^m$ such that for each $i \in \{1, \ldots, m\}$, we have:
\begin{itemize}
\item[$(a)$] If $\bar{\eta}_{i}=0$, then $(\bar{x}^{i}, \bar{y}^{i}) = (0,0)$
and the associated KKT condition is given by 
\begin{equation}  \label{KKT_etai_0}
 \begin{array}{l}
 \alpha_{i} (\bar{P}^{i})^{\top} \geq \nabla u(\mathbf{0}) A^{i} \geq \gamma_{i} \alpha_{i} (\bar{P}^{i})^{\top} - (\mu^{i})^{\top}, \\ 
  \hspace{1.25cm} q_{i} - (\mu^{i})^{\top} b^{i} \geq 0.
 \end{array}
\end{equation}

\item[$(b)$] If $\bar{\eta}_{i} \neq 0$, then the associated KKT condition is given by 
\begin{equation}
\!\!\!\!\!\!\!\!\!\!\!\!
\begin{array}{ll}
 \alpha_{i} \bar{P}_{j}^{i} \geq \sum_{k=1}^{n} D_{k} u(A(\mathbf{\bar{x}} - \mathbf{\bar{y}})) A_{kj}^{i} \geq \gamma_{i} \alpha _{i} \bar{P}_{j}^{i} + \mu_{j}^{i}, & \!\!\!\! \text{when } \bar{x}_{j}^{i}=\bar{y}_{j}^{i}=0,\medskip  \\ 
\sum_{k=1}^{n}D_{k}u(A(\mathbf{\bar{x}}-\mathbf{\bar{y}}))A_{kj}^{i}=\gamma
_{i}\alpha _{i}\bar{P}_{j}^{i}+\mu _{j}^{i}, & \!\!\!\! \text{when } \bar{x}%
_{j}^{i}=0,\bar{y}_{j}^{i}>0,\medskip  \\ 
\sum_{k=1}^{n}D_{k}u(A(\mathbf{\bar{x}}-\mathbf{\bar{y}}))A_{kj}^{i}=\alpha
_{i}\bar{P}_{j}^{i}, & \!\!\!\! \text{when } 0<\bar{x}_{j}^{i}<R_{j}^{i},\text{ }%
\bar{y}_{j}^{i}=0, \medskip  \\ 
\sum_{k=1}^{n}D_{k}u(A(\mathbf{\bar{x}}-\mathbf{\bar{y}}))A_{kj}^{i}\geq
\alpha _{i}\bar{P}_{j}^{i}, & \!\!\!\! \text{when } \bar{x}_{j}^{i}=R_{j}^{i}\text{, 
}\bar{y}_{j}^{i}=0,\medskip  \\ 
(\mu ^{i})^{\top} \left( \bar{y}^{i}-\bar{\eta}_{i}b^{i}\right) =0, & 
\medskip \label{KKT_etai_1} \\ 
q_{i}-(\mu ^{i})^{\top} b^{i}=0, & \medskip  \\ 
\varphi _{i}(R^{i}+\gamma _{i}{\bar{y}^{i}}-{\bar{x}^{i}})=\varphi_{i}(R^{i}). & 
\end{array}
\end{equation}
\end{itemize}
\end{thr}

\begin{proof}
By Proposition \ref{lm:KRZ}, the constraint qualification (KRZ) holds and, applying \cite[Theorem 5.6]{jahn2020introduction}, it follows that the system \eqref{KKT-1} holds for $(\mathbf{\bar{x}}, \mathbf{\mathbf{\bar{y}},} \bar{\eta})$. Hence, there there exist $(\mu^{1}, \ldots, \mu^{m}) \in \prod_{i=1}^{m} \mathbb{R}_{+}^{n_{i}}$ and $\lambda \in \mathbb{R}^{m}$ such that 
\begin{subequations}
\begin{flalign}
 & \sum\limits_{i=1}^{m} \bigg[ \nabla u(A(\mathbf{\bar{x}} - \mathbf{\bar{y}}))A^{i} ( -( x^{i} - \bar{x}^{i}) + ( y^{i} - \bar{y}^{i}) ) + q_{i} (\eta_{i} - \bar{\eta}^{i}) \notag \\
 & \quad (\mu^{i})^{\top} \left[ ( y^{i} - \bar{y}^{i}) - (\eta_{i} - \bar{\eta}_{i}) b^{i} \right] + \lambda_{i} (\bar{P}^{i})^{\top} (-( x^{i} - \bar{x}^{i}) + \gamma_{i} (y^{i} - \bar{y}^{i} )) \bigg] \geq 0, \notag \\
 & (\mu^{i})^{\top} ( \bar{y}^{i} -\bar{\eta}_{i} b^{i}) = 0, \quad \text{for every } i=1,\ldots,m, \label{KKT_system} \\
 & \varphi_{i} (R^{i} + \gamma_{i} \bar{y}^{i} - \bar{x}^{i}) = \varphi_{i} (R^{i}) \quad \text{for every } i=1,\ldots,m, \notag \\
 & \text{and } (\mathbf{x}, \mathbf{y}, \eta) \in \mathbf{\hat{S}} \text{ arbitrary.} \notag 
\end{flalign}
\end{subequations}
This is equivalent to 
\begin{subequations}
\label{KKT}
\begin{flalign}
& \sum_{i=1}^{m} \bigg[ (-\nabla u(A(\mathbf{\bar{x}} - \mathbf{\bar{y}}))A^{i} - \lambda_{i} (\bar{P}^{i})^{\top}) \left(x^{i}-\bar{x}^{i} \right) \bigg] + \sum_{i=1}^{m} \bigg[ (q_{i} - (\mu^{i})^{\top} b^{i}) \left( \eta _{i}-\bar{\eta}_{i} \right) \bigg] \nonumber \\
& \quad + \sum_{i=1}^{m} \bigg[ (\nabla u(A(\mathbf{\bar{x}} - \mathbf{\bar{y}})) A^{i} + \gamma_{i} \lambda_{i} (\bar{P}^{i})^{\top} + (\mu^{i})^{\top}) \left( y^{i} - \bar{y}^{i} \right) \bigg] \geq 0, \medskip  \label{KKT_a} \\
& (\mu^{i})^{\top} \left( \bar{y}^{i} - \bar{\eta}_{i} b^{i} \right) = 0, \quad \text{for every } i=1,\ldots,m, \medskip  \label{KKT_b} \\
& \varphi_{i} (R^{i} + \gamma_{i} \bar{y}^{i} - \bar{x}^{i}) = \varphi_{i} (R^{i}) \quad \text{for every } i=1,\ldots,m, \nonumber \\
& \text{and } (\mathbf{x},\mathbf{y},\eta) \in \mathbf{\hat{S}} \text{ arbitrary.}  \label{KKT_c}
\end{flalign}
\end{subequations}

Since the first inequality holds for all $x^{i}, y^{i} \in \mathbb{R}_{+}^{n_{i}}$ and all $\eta \in [0,1]^{m}$, the previous conditions can be equivalently studied on each market. Therefore, for each mar\-ket $i=1,\ldots,m$, the following conditions must hold:
\begin{flalign}
& \left( - \nabla u(A(\mathbf{\bar{x}} - \mathbf{\bar{y}})) A^{i} - \lambda_{i} (\bar{P}^{i})^{\top} \right) \left( x^{i} - \bar{x}^{i} \right) \geq 0 \quad \text{for all } x^{i} \geq 0,  \label{eq:01} \\
& \left( \nabla u(A(\mathbf{\bar{x}} - \mathbf{\bar{y}})) A^{i} + \gamma_{i} \lambda_{i} (\bar{P}^{i})^{\top} + (\mu^{i})^{\top} \right) \left( y^{i} - \bar{y}^{i} \right) \geq 0\quad \text{for all }y^{i} \geq 0,  \label{eq:02} \\
& \left( q_{i} - (\mu^{i})^{\top} b^{i} \right) \left( \eta_{i} - \bar{\eta}_{i} \right)
\geq 0 \quad \text{for all } \eta_{i} \geq 0,  \label{eq:03} \\
& ~ (\mu^{i})^{\top} (\bar{y}^{i} - \bar{\eta}_{i} b^{i}) = 0 , \label{eq:04}
\end{flalign}
and equivalently on each asset for the first two equations,
\begin{flalign}
& \left( -\sum_{k=1}^{n} D_{k}u((A(\mathbf{\bar{x}} - 
\mathbf{\bar{y}})) A_{kj}^{i} - \lambda_{i} \bar{P}_{j}^{i} \right) \left(x_{j}^{i} - \bar{x}_{j}^{i} \right) \geq 0 \quad \text{for all } x_{j}^{i} \geq 0,
\label{KKT_eq_a} \\
&   \left( \sum_{k=1}^{n} D_{k}u (A(\mathbf{\bar{x}} - \mathbf{\bar{y}})) A_{kj}^{i} + \gamma_{i} \lambda_{i} \bar{P}_{j}^{i} + \mu_{j}^{i} \right) \left( y_{j}^{i} - \bar{y}_{j}^{i} \right) \geq 0 ~ \text{ for all } y_{j}^{i} \geq 0, \label{KKT_eq_b} \\
& \, \left( q_{i} - (\mu^{i})^{\top} b^{i} \right) \left( \eta _{i} - \bar{\eta}_{i} \right) \geq 0 ~ \text{for all }\eta_{i} \geq 0,  \label{KKT_eq_c} \\
& ~ (\mu^{i})^{\top} (\bar{y}^{i} - \bar{\eta}_{i} b^{i}) = 0,  \label{KKT_eq_d}
\end{flalign}
for every $j = 1,\ldots,n_{i}$. 

In general, if $\bar{\eta}_{i} > 0$, then $\eta _{i}-\bar{\eta}_{i}$ can be   indistinctly strictly positive or strictly negative so, from \eqref{KKT_eq_c} and \eqref{KKT_eq_d}, we get 
\begin{align}
 & q_{i} - (\mu^{i})^{\top} b^{i} = 0, \label{mu_KKT_eq} \\
 & \hspace{0.6cm} \, (\mu^{i})^{\top} \bar{y}^{i} = \bar{\eta}_{i}  (\mu^{i})^{\top} b^{i}. \label{mu_KKT_eq_2}
\end{align}

On the other hand, if $\bar{\eta}_{i}=0$, then from \eqref{KKT_eq_c} we get $(q_{i} - (\mu^{i})^{\top} b^{i}) \eta_{i} \geq 0$ for all $\eta_{i} \geq 0$, thus \eqref{KKT_eq_c} is equivalent to
\begin{equation}\label{mu_KKT_ineq}
 q_{i} - (\mu^{i})^{\top} b^{i} \geq 0.
\end{equation}
Therefore, we can distinguish two cases based on the activation of the market. \\

\noindent $(i)$: When $\bar{\eta}_{i} = 0$, i.e., the market $i$ is inactive, thus Remark \ref{rem:mu} implies that $\bar{x}^{i} = \bar{y}^{i} = 0$. As previously discussed, condition \eqref{KKT_eq_c} is equivalent to \eqref{mu_KKT_eq},  while condition \eqref{KKT_eq_d} is automatically satisfied.

On the other hand, conditions \eqref{KKT_eq_a} and \eqref{KKT_eq_b} are equivalent to
\begin{equation*}
\begin{array}{c} 
- \nabla u(\mathbf{0}) A^{i} - \lambda_{i} (\bar{P}^{i})^{\top} \geq 0, \\ 
\nabla u(\mathbf{0}) A^{i} + \gamma_{i} \lambda_{i} (\bar{P}^{i})^{\top} + (\mu^{i})^{\top} \geq 0.
\end{array}
\label{200325}
\end{equation*}
From the first inequality, we have $-\lambda_{i} (\bar{P}^{i})^{\top} \geq \nabla u(\mathbf{0}) A^{i}$, which implies:
\begin{equation*}
-\lambda_{i} \bar{P}_{j}^{i} \geq \sum_{k=1}^{n} D_{k}u(\mathbf{0}) A_{kj}^{i} \quad \text{for every } j \in \{1,\ldots,n_{i}\}.
\end{equation*}

Since $D_{k}u(0) \geq 0$, $\bar{P}_{j}^{i} \geq 0$ by hypotheses \eqref{p:utility} and \eqref{p:market}, and $A_{kj}^{i} \geq 0$ by de\-fi\-ni\-tion, it follows that $\lambda_{i}\leq 0$. Consequently, \eqref{200325} can be rewritten as
\begin{equation*}
|\lambda _{i}| (\bar{P}^{i})^{\top} \geq \nabla u(\mathbf{0}) A^{i} \geq \gamma_{i} |\lambda_{i}| (\bar{P}^{i})^{\top} - (\mu^{i})^{\top}.
\end{equation*}
This proves that the KKT system \eqref{KKT_system} is equivalent to  \eqref{KKT_etai_0}. \\

\noindent $(ii)$: When $\bar{\eta}_{i} \not = 0$, we distinguish traded assets indices. For each $j \in \{1,\ldots,n_{i}\}$  and based on the complementary condition \eqref{def:comple}, we have the fo\-llo\-wing cases:
\begin{itemize}
\item If $\bar{x}_{j}^{i} = \bar{y}_{j}^{i} = 0$, then
\begin{subequations}
\label{KKT_eq}
\begin{eqnarray*}
-\sum_{k=1}^{n} D_{k}u(A(\mathbf{\bar{x}} - \mathbf{\bar{y}})) A_{kj}^{i} - \lambda_{i} \bar{P}_{j}^{i} &\geq& 0 \text{, } \\ \sum_{k=1}^{n}   D_{k}u(A(\mathbf{\bar{x}} - \mathbf{\bar{y}})) A_{kj}^{i} + \gamma_{i} \lambda_{i} \bar{P}_{j}^{i} + \mu_{j}^{i} &\geq& 0 \text{, }
\end{eqnarray*}
and, following the same reasoning as before, that is equivalent to  
\end{subequations}
\begin{equation}\label{C1}
\lvert \lambda_{i} \rvert \bar{P}_{j}^{i} \geq \sum_{k=1}^{n} D_{k}u((A(\mathbf{\bar{x}} - \mathbf{\bar{y}})) A_{kj}^{i}\geq \gamma _{i} \lvert \lambda_{i} \rvert \bar{P}_{j}^{i} + \mu_{j}^{i}.  
\end{equation}

\item If $\bar{x}_{j}^{i} = 0$ and $\bar{y}_{j}^{i} > 0$, then \eqref{KKT_eq_a} and \eqref{KKT_eq_b} are equivalent to 
\begin{eqnarray}\label{1903254}
 - \sum_{k=1}^{n} D_{k}u(A(\mathbf{\bar{x}} - \mathbf{\bar{y}})) A_{kj}^{i} - \lambda_{i} \bar{P}_{j}^{i} &\geq& 0\text{, } \\ \nonumber \sum_{k=1}^{n} D_{k}u(A(\mathbf{\bar{x}} - \mathbf{\bar{y}})) A_{kj}^{i} + \gamma_{i} \lambda_{i} \bar{P}_{j}^{i} + \mu_{j}^{i} &=& 0.
\end{eqnarray}
As before, from the first equation we deduce that $\lambda_{i} \leq 0$. Furthermore, the inequality follows from the equality since
\begin{equation*}
 \sum_{k=1}^{n} D_{k}u(A(\mathbf{\bar{x}} - \mathbf{\bar{y}})) A_{kj}^{i} = \gamma_{i} \lvert \lambda_{i} \rvert \bar{P}_{j}^{i} - \mu_{j}^{i} \leq \gamma_{i} \lvert \lambda_{i} \rvert \bar{P}_{j}^{i} \leq \lvert \lambda_{i} \rvert \bar{P}_{j}^{i} = - \lambda_{i} \bar{P}_{j}^{i},
\end{equation*}
thus \eqref{1903254} is equivalent to 
\begin{equation}\label{C2}
 \sum_{k=1}^{n}D_{k}u(A(\mathbf{\bar{x}} - \mathbf{\bar{y}})) A_{kj}^{i} + \gamma_{i} \lambda_{i} \bar{P}_{j}^{i} + \mu_{j}^{i}=0.
\end{equation}

\item If $R_{j}^{i} > \bar{x}_{j}^{i}>0$ and $\bar{y}_{j}^{i}=0$, then \eqref{KKT_eq_a} and \eqref{KKT_eq_b} are equivalent to 
\begin{eqnarray}\label{1903256}
 - \sum_{k=1}^{n} D_{k}u(A(\mathbf{\bar{x}} - \mathbf{\bar{y}})) A_{kj}^{i} - \lambda_{i} \bar{P}_{j}^{i} &=& 0\text{, } \\ \nonumber \sum_{k=1}^{n} D_{k}u(A(\mathbf{\bar{x}} - \mathbf{\bar{y}})) A_{kj}^{i} + \gamma_{i} \lambda_{i} \bar{P}_{j}^{i} + \mu_{j}^{i} &\geq& 0. \notag
\end{eqnarray}

Using \eqref{p:utility} in the first equation, we deduce that $\lambda_{i}\leq 0$, and 
\begin{equation}
 \sum_{k=1}^{n} D_{k}u(A(\mathbf{\bar{x}} - \mathbf{\bar{y}})) A_{kj}^{i} = \lvert \lambda_{i} \rvert \bar{P}_{j}^{i} \geq \gamma_{i} \lvert \lambda_{i} \rvert \bar{P}_{j}^{i} \geq \gamma_{i} \lvert \lambda_{i} \rvert \bar{P}_{j}^{i} - \mu_{j}^{i}. 
\label{C3}
\end{equation}

\item If $\bar{x}_{j}^{i} = R_{j}^{i}$ and $\bar{y}_{j}^{i} = 0$, then this situation cannot hold for every asset, that is, there exists $k \neq j$ such that $x_{k}^{i} = R_{k}^{i}$. Indeed, if otherwise $\bar{x}^{i} = R^{i}$, then $\varphi_{i} (R^{i}) = \varphi_{i} (0)$, which contradicts \eqref{p:market}. This implies $\lambda_{i} \leq 0$. 

Since in this case, \eqref{KKT_eq_a} and \eqref{KKT_eq_b} are equivalent to 
\begin{eqnarray}\nonumber
- \sum_{k=1}^{n} D_{k}u(A(\mathbf{\bar{x}}-\mathbf{\bar{y}})) A_{kj}^{i} - \lambda_{i} \bar{P}_{j}^{i} &\leq& 0, \\ \nonumber \sum_{k=1}^{n} D_{k}u(A(
\mathbf{\bar{x}} - \mathbf{\bar{y}})) A_{kj}^{i} + \gamma_{i} \lambda_{i} \bar{P}_{j}^{i} + \mu_{j}^{i} &\geq& 0,
\end{eqnarray}
then, the condition reduces to 
\begin{equation}\label{C4}
 \sum_{k=1}^{n} D_{k}u(A(\mathbf{\bar{x}} -\mathbf{\bar{y}})) A_{kj}^{i} \geq |\lambda _{i}| \bar{P}_{j}^{i}.
\end{equation}
\end{itemize}

Finally, relation \eqref{KKT_etai_1} follows from \eqref{C1}, \eqref{C2}, \eqref{C3}, \eqref{C4}, and condition \eqref{KKT_eq_a}.
\end{proof}

We analyze the KKT conditions of Theorem \ref{nec:KKT} below.

\begin{remark} \label{rem:yi}
Some useful insights can be derived from the KKT conditions. Consider, for instance, part {\it (a)} of Theorem \ref{nec:KKT}. Under the same assumptions as in this theorem, the optimality system yields the following bound for the multipliers:
\begin{equation} \label{eq:multiplier_bound}
 \mu_{j}^{i} \leq \frac{q_{i}}{b_{j}^{i}}, ~ \forall ~ j \in \{1,\dots,n_{i}\}, ~ \forall ~ i \in \{1,\dots,m\}.
\end{equation} 
The core result of part {\it (a)} of Theorem \ref{nec:KKT} is equation \eqref{KKT_etai_0}, which characterizes the optimality of refraining from trading entirely. From this condition, it becomes clear that the no-trade region expands with larger positive values of the multipliers. Moreover, inequality \eqref{eq:multiplier_bound} shows that each multiplier is bounded above by a constant proportional to the gas fee and inversely proportional to the activation bound. Therefore, the gas fees enlarge the no-trading region, which is a reasonable effect. Further discussion on the conditions that favor no trading will be continued in Section \ref{sec:04}.

 Furthermore, part {\it (b)} of Theorem \ref{nec:KKT} corresponds to the case $\bar{\eta}_{i} > 0$. There, the KKT condition $(\mu^{i})^{\top} (\bar{y}^{i} - \bar{\eta}_{i} b^{i}) = 0$ is equi\-va\-lent to
\begin{equation*}
 \mu_{j}^{i} \bar{y}_{j}^{i} = \mu_{j}^{i} \bar{\eta}_{i} b_{j}^{i},
\end{equation*}
Consequently, the optimal trades associated with strictly positive
multipliers are determined, 
\begin{equation}
 \mu _{{j}}^{i}>0 ~ \Longrightarrow ~ \bar{y}_{{j}}^{i} = \bar{\eta}_{{i}} b_{{j}}^{i}.
\label{eq:posmu}
\end{equation} 
\end{remark}

\begin{remark} \label{rem:KRZ}
 The technical assumption \eqref{p:y} has important financial consequences:
 \begin{itemize}
\item[$(i)$] {\rm Integer Solution Penalty:} The condition $\mathbf{b}-\mathbf{\bar{y}} > 0$ in \eqref{p:y} penalizes integer solutions when $\bar{\eta}_{i},$ $q_{i}>0$, since it necessarily implies $\bar{\eta}_{i} \neq 1$ for all $i \in \{1,\ldots ,m\}$. In this sense, under assumption \eqref{p:y}, every solution $(\mathbf{\bar{x}},\mathbf{\bar{y}}, \bar{\eta}) \in \mathbf{\hat{S}}$ of problem \eqref{rGF} is a KKT point by Theorem \ref{nec:KKT}. Therefore, if $\bar{\eta}_{i} > 0$, it follows from \eqref{eq:posmu} that $\bar{\eta}_{{i}} \neq 1$ (unless $\mu ^{i} = 0$). In this case $q_{i}=0$ by condition \eqref{KKT_etai_1}, which corresponds to a market with no gas fee. Hence, for a market with a gas fee and trade, necessarily $\bar{\eta}_{{i}} \neq 1$ and at least one trade satisfies $\bar{y}_{{j}}^{i} = \bar{\eta}_{i} b_{j}^{i}$. However, let us emphasize that this does not necessarily imply that all positive trades are saturated, i.e., $\bar{y}_{{j}}^{i} = \bar{\eta}_{{i}}b_{{j}}^{i}$; in fact, numerical examples show situations with trades $\bar{y}_{{j}}^{i} \in (0, \bar{\eta}_{i} b_{{j}}^{i})$. 

\item[$(ii)$] {\rm Tender Size Limit:} The condition $b_j^i - \bar{y}_j^i \leq \frac{2}{\gamma_i} R_j^i$ imposes an upper bound on the maximum tender size. This holds particularly when $b^i \leq \frac{2}{\gamma_i} R^i$.
\end{itemize}
\end{remark}

\subsection{Sufficient Optimality Conditions}\label{subsec:3-2}

We provide sufficient conditions without convexity a\-ssump\-tions neither on the utility function $u$ nor the trade functions $\varphi _{i}$. To that end, let $K \subseteq \mathbb{R}^{n}$ and $C \subseteq \mathbb{R}^{m}$ be two nonempty sets and $g: \mathbb{R}^{n} \rightarrow \mathbb{R}^{m}$ be a differentiable mapping. 

It is said that $g$ is $C$-quasiconvex at $\overline{x} \in K$ with respect to $K$ if for all $x\in K$ (see \cite[Definition 5.12]{jahn2020introduction}), the following implication holds:
\begin{equation}\label{C:quasiconvex}
 g(x) - g(\bar{x}) \in C ~ \Longrightarrow ~ \nabla g(\bar{x}) (x - \bar{x}) \in C. 
\end{equation}

A sufficient optimality condition is given in the next result.

\begin{thr} \label{theo:suff:KKT}
 Assume that $f$ is pseudoconcave (thus $-f$ is pseudoconvex) and $\varphi_{i}$ is quasilinear for every $i=1,\ldots,m$. If $(\mathbf{\bar{x}}, \mathbf{\bar{y}}, \bar{\eta}) \in \mathbf{\hat{S}}$ verifies optimality system \eqref{KKT_etai_0} and \eqref{KKT_etai_1}, then it solves problem \eqref{rP}.
\end{thr}

\begin{proof}
 We apply \cite[Corollary 5.15]{jahn2020introduction}. In first place, given the set
\begin{equation}
 \widehat{C} := (\mathbb{R}_{-} \backslash \{0\}) \times \! \left( \! - \! \left( \prod\limits_{i=1}^{m} \mathbb{R}_{+}^{n_{i}} \! \right) \! + \cone(\{G(\mathbf{\bar{x}}, \mathbf{\bar{y}}, \bar{\eta})\}) - \cone(\{G(\mathbf{\bar{x}}, \mathbf{\bar{y}}, \bar{\eta})\}) \! \right) \times \{0\}, \notag
\end{equation}
let us prove that $(-f,G,H)$ is $\widehat{C}$-quasiconvex at $(\mathbf{\bar{x}}, \mathbf{\bar{y}}, \bar{\eta})$ with respect to $\mathbf{\hat{S}}$. That is, we will prove that for every $\left( \mathbf{x}, \mathbf{y}, \eta \right) \in \mathbf{\hat{S}}$, we have
\begin{align}
 & \hspace{0.9cm} (- f, G, H) (\mathbf{x}, \mathbf{y}, \eta) - (- f, G, H) (\mathbf{\bar{x}}, \mathbf{\bar{y}}, \bar{\eta}) \in \widehat{C} \notag \\ 
 & \Longrightarrow ~ \nabla (-f, G, H)(\mathbf{\bar{x}}, \mathbf{\bar{y}}, \bar{\eta}) \left( \mathbf{x - \bar{x}}, \mathbf{y-\bar{y}}, \eta - \bar{\eta} \right) \in \widehat{C}. \label{for:suff}
\end{align}

Relation \eqref{for:suff} can be expressed by components as follows:
\begin{itemize}
 \item[$(i)$] If $-f(\mathbf{x}, \mathbf{y}, \eta) + f(\mathbf{\bar{x}}, \mathbf{\bar{y}}, \bar{\eta}) \in \mathbb{R}_{-} \backslash \{0\}$, then
 $$-\nabla f(\mathbf{\bar{x}}, \mathbf{\bar{y}}, \bar{\eta}) \left( \mathbf{x - \bar{x}}, \mathbf{y - \bar{y}}, \eta - \bar{\eta} \right) \in \mathbb{R}_{-} \backslash \{0\}.$$

 \item[$(ii)$] If
 \begin{eqnarray}\nonumber
 G_{i} (\mathbf{x}, \mathbf{y},\eta) - G_{i} (\mathbf{\bar{x}}, \mathbf{\bar{y}}, \bar{\eta}) \in \\ \nonumber
 -  \left( \prod\limits_{i=1}^{m} \mathbb{R}_{+}^{n_{i}} \right)  + \cone(\{G_{i}(\mathbf{\bar{x}}, \mathbf{\bar{y}}, \bar{\eta})\}) - \cone(\{G_{i} (\mathbf{\bar{x}}, \mathbf{\bar{y}}, \bar{\eta})\}),
 \end{eqnarray}
 then
 \begin{eqnarray}\nonumber
 \nabla G_{i} (\mathbf{\bar{x}}, \mathbf{\bar{y}}, \bar{\eta}) (\mathbf{x-\bar{x}}, \mathbf{y-\bar{y}}, \eta - \bar{\eta}) \in \\ \nonumber
 - \left( \prod\limits_{i=1}^{m} \mathbb{R}_{+}^{n_{i}} \right) + \cone(\{G_{i}(\mathbf{\bar{x}}, \mathbf{\bar{y}}, \bar{\eta})\}) - \cone(\{G_{i}(\mathbf{\bar{x}}, \mathbf{\bar{y}}, \bar{\eta})\})
 \end{eqnarray}
 for every $i=1,\ldots,m$.

 \item[$(iii)$] If $H_{i} (\mathbf{x}, \mathbf{y}, \eta) = H_{i} (\mathbf{\bar{x}}, \mathbf{\bar{y}}, \bar{\eta})$, then $\nabla H_{i} (\mathbf{\bar{x}},\mathbf{\bar{y}}, \bar{\eta}) (\mathbf{x} - {\bf \bar{x}}, \mathbf{y} - {\bf \bar{y}}, \eta -\bar{\eta}) = 0$ for every $i=1,\ldots,m$.
\end{itemize}

Then, let us prove the previous conditions one by one: 
\begin{itemize}
 \item[$(i)$:] Let $(\mathbf{x}, \mathbf{y}, \eta) \in \mathbf{\hat{S}}$ ($\mathbf{x} \neq \mathbf{y}$ always). Then,  since $f$ is pseudoconcave  
 \begin{align*}
-f(\mathbf{x}, \mathbf{y}, \eta ) +  f(\mathbf{\bar{x}}, \mathbf{\bar{y}}, \bar{\eta}) & \, \in -\mathbb{R}_{-} \backslash \{0\} 
 \Longleftrightarrow - f(\mathbf{\bar{x}}, \mathbf{\bar{y}}, \bar{\eta}) > - f(\mathbf{x}, \mathbf{y}, \eta) \\
 & \overset{\eqref{pseudoconvex}}{\Longrightarrow} - \nabla f(\mathbf{\bar{x}}, \mathbf{\bar{y}}, \bar{\eta}) (\mathbf{x-\bar{x}}, \mathbf{y-\bar{y}}, \eta - \bar{\eta}) < 0 \\
 & \Longleftrightarrow - \nabla f(\mathbf{\bar{x}}, \mathbf{\bar{y}}, \bar{\eta}) (\mathbf{x-\bar{x}}, \mathbf{y-\bar{y}}, \eta -\bar{\eta}) \in -\mathbb{R}_{-} \backslash \{0\},
\end{align*}
thus $(i)$ holds.

 \item[$(ii)$:] Let $(\mathbf{x}, \mathbf{y}, \eta) \in \mathbf{\hat{S}}$. Since $G_{i}$ is convex (affi\-ne), we have
 \begin{align*}
 \nabla G_{i}(\mathbf{\bar{x}}, \mathbf{\bar{y}}, \bar{\eta})(\mathbf{x-\bar{x}}, \mathbf{y-\bar{y}}, \eta -\bar{\eta}) & = G_{i} (\mathbf{x}, \mathbf{y}, \eta) - G_{i} (\mathbf{\bar{x}}, \mathbf{\bar{y}}, \bar{\eta}) \\ 
 & = y^{i} - \bar{y}^{i} - \left( \eta _{i} - \bar{\eta}_{i} \right) b^{i},
\end{align*}
for every $(\mathbf{x}, \mathbf{y}, \eta) \in \mathbf{\hat{S}}$, thus the
implication is obvious.

 \item[$(iii)$:] Let $(\mathbf{x}, \mathbf{y}, \eta) \in \mathbf{\hat{S}}$. Then, applying the quasilinearity of $\varphi _{i}$, we have
 \begin{equation*}
 \begin{aligned}
 & H_{i} (\mathbf{x}, \mathbf{y}, \eta) = H_{i} (\mathbf{\bar{x}}, \mathbf{\bar{y}}, \bar{\eta}) \, \Longleftrightarrow
 \varphi_{i} (R^{i}+ \gamma_{i}y^{i}-x^{i}) = \varphi_{i} (R^{i}+ \gamma_{i} \bar{y}^{i} - \bar{x}^{i}) \\
 & \overset{\eqref{char:quasilinear}}{\Longrightarrow} \nabla \varphi_{i} (R^{i} + \gamma_{i} \bar{y}^{i} - \bar{x}^{i}) (\gamma_{i} (y^{i} - \bar{y}^{i}) - (x^{i} -\bar{x}^{i})) = 0 \\
 & \Longleftrightarrow
 \gamma_{i} \nabla \varphi_{i}(R^{i} + \gamma_{i} \bar{y}^{i} - \bar{x}^{i}) (y^{i} - \bar{y}^{i})
 - \nabla \varphi_{i}(R^{i} + \gamma_{i} \bar{y}^{i} - \bar{x}^{i}) (x^{i} - \bar{x}^{i})= 0 \\
 & \Longleftrightarrow \nabla H_{i}(\mathbf{\bar{x}}, \mathbf{\bar{y}}, \bar{\eta}) (\mathbf{x-\bar{x}}, \mathbf{y-\bar{y}}, \eta - \bar{\eta}) = 0,
 \end{aligned}
\end{equation*}
and $(iii)$ holds.
\end{itemize}

Therefore, $(-f,G,H)$ is $\widehat{C}$-quasiconvex at $(\mathbf{\bar{x}, \bar{y},} \bar{\eta})$ and, consequently, if KKT conditions \eqref{KKT_etai_0} and \eqref{KKT_etai_1} are verified for that point, then the point $(\mathbf{\bar{x}, \bar{y},} \bar{\eta})$ is op\-ti\-mal for problem \eqref{rP} by \cite[Corollary 5.15]{jahn2020introduction}. 
\end{proof}

As a consequence of Theorems \ref{nec:KKT} and \ref{theo:suff:KKT} we can establish a general cha\-rac\-terization of optimal points for problem \eqref{rGF}.

\begin{thr} \label{charact:KKT}
 Assume that properties \eqref{p:Rb}, \eqref{p:utility}, and \eqref{p:market} are verified, $f$ is pseudoconcave, and $\varphi_{i}$ is quasilinear for every $i=1,\ldots,m$. Suppose that $(\mathbf{\bar{x}}, \mathbf{\bar{y}}, \bar{\eta}) \in \mathbf{\hat{S}}$ satisfies \eqref{def:comple} and \eqref{p:y}. Then, the following assertions are equi\-va\-lent:
\begin{itemize}
 \item[$(a)$] $(\mathbf{\bar{x}}, \mathbf{\bar{y}}, \bar{\eta})$ solves problem \eqref{rP}.

 \item[$(b)$] $(\mathbf{\bar{x}}, \mathbf{\bar{y}}, \bar{\eta})$ verifies the optimality system \eqref{KKT_etai_0} and \eqref{KKT_etai_1}.
\end{itemize}
\end{thr}

  In the next remark, we analyze the scope of the approximate problem $(\mathcal{P}^{r} (\mathbf{b}, q))$ in relation to problem $(\mathcal{P}(\mathbf{b},q))$. 
\begin{remark} \label{rem:comparison}

 The approximate solution $(\mathbf{\bar{x}},\mathbf{\bar{y}},\bar{\eta})\in 
\mathbf{S}$ to the relaxed problem $(\mathcal{P}^{r}(\mathbf{b},q))$ differs
from a solution $(\mathbf{\tilde{x}},\mathbf{\tilde{y}},\tilde{\eta})$ of
the original problem $(\mathcal{P}(\mathbf{b},q))$ in different aspects. For
instance, this approximation could not incorporate the fixed gas fee penalty because $\bar{\eta}_{i}$ does not necessarily belong to $\{0,1\}$. In fact, if equation \eqref{p:y} is satisfied, then we saw in Remark \ref{rem:KRZ} that $\bar{\eta}_{i}<1$ for all $i\in\{1,\ldots,m\}$. In this case, the associated fee $\sum_{i=1}^{m}q_{i}\bar{\eta}_{i}$ does not match the true fee $\sum_{\tilde{\eta}_{i}>0}q_{i}$, which can lead to utility trade-offs.   

To properly compare these solutions, let us consider any feasible point 
\begin{equation*}
(\mathbf{x},\mathbf{y},\eta )\in \prod_{i=1}^{m}[0,R^{i}]\times
\prod_{i=1}^{m}[0,b^{i}]\times \left\{ 0,1\right\} ^{m}
\end{equation*}%
of $(\mathcal{P}(\mathbf{b},q))$, which by construction is also feasible for 
$(\mathcal{P}^{r}(\mathbf{b},q))$. For such points, the following inequality
holds: 
\begin{equation*}
-u(A(\mathbf{\bar{x}-\bar{y}}))+\sum_{i=1}^{m}q_{i}\bar{\eta}_{i}\leq -u(A(%
\mathbf{x-y}))+\sum_{i=1}^{m}q_{i}\eta _{i}.
\end{equation*}%
That is, 
\begin{align*}
-u(A(\mathbf{\bar{x}-\bar{y}}))& \,-(-u(A(\mathbf{x-y})))\leq
\sum_{i=1}^{m}q_{i}\left( \eta _{i}-\bar{\eta}_{i}\right) =\sum_{\eta
_{i}>0}q_{i}\left( 1-\bar{\eta}_{i}\right) -\sum_{\eta _{i}=0}q_{i}\bar{\eta}%
_{i} \\
\ & \leq q_{\max }\sum_{\eta _{i}>0}1-q_{\min }\sum_{i=1}^{m}\bar{\eta}%
_{i}\,=\,q_{\max }\Vert \eta \Vert _{0}-q_{\min }\Vert \bar{\eta}\Vert _{1}
\\
& =q_{\max }(\Vert \eta \Vert _{0}-\Vert \bar{\eta}\Vert _{1})+(q_{\max
}-q_{\min })\Vert \bar{\eta}\Vert _{1},
\end{align*}%
where $q_{\max }=\underset{i=1,\ldots ,m}{\max }q_{i}$, $q_{\min }=\underset{%
i=1,\ldots ,m}{\min }q_{i}$, $\Vert \eta \Vert _{0}=\#\{i:\eta _{i}>0\}$ is
the $\ell _{0}$ pseudo-norm, and $\Vert \eta \Vert _{1}=\sum_{i=1}^{m}|\eta
_{i}|$ is the $\ell _{1}$-norm. In terms of maximizing the utility, this
corresponds to 
\begin{equation}
u(A(\mathbf{\bar{x}-\bar{y}}))\geq u(A(\mathbf{x-y}))-(q_{\max }(\Vert \eta
\Vert _{0}-\Vert \bar{\eta}\Vert _{1})+(q_{\max }-q_{\min })\Vert \bar{\eta}%
\Vert _{1}).  \label{eq:general_bound}
\end{equation}%
This holds for every $(\mathbf{x}, \mathbf{y}, \eta )$. In particular, we can take $(\mathbf{x}, \mathbf{y})$ to be the solution to the following utility maximization problem, which considers the active markets given by $\eta $, that is,
\begin{equation}
\begin{aligned} &\text{maximize} && u(\Psi) \\ &\text{subject to} && \Psi =
\sum_{i=1}^{m} A^{i} (x^{i} - y^{i}), \\ &&& y^{i} \leq \mathbf{1}_{\{\eta >
0\}} b^{i}, \\ &&& \varphi_{i}(R^{i} + \gamma_{i} y^{i} - x^{i}) =
\varphi_{i}(R^{i}), \\ &&& 0 \leq x^{i} \leq R^{i}, \quad 0 \leq y^{i}, \
\forall ~ i \in \{1, \ldots, m\}, \\ \end{aligned} 
\tag{${\mathcal{P}}^{\eta
}({\mathbf{b}},q)$}  \label{P_eta}
\end{equation}%
where 
\begin{equation*}
1_{\{\eta >0\}}:=\left\{ 
\begin{array}{cc}
1 & \eta _{i}>0, \\ 
0 & \text{otherwise.}%
\end{array}%
\right. 
\end{equation*}%
If $\varepsilon =q_{\max }(\Vert \eta \Vert _{0}-\Vert \bar{\eta}\Vert_{1}) + (q_{\max} - q_{\min }) \Vert \bar{\eta}\Vert _{1}>0$, then $(\mathbf{\bar{x}}, \mathbf{\bar{y}})$ corresponds to an $\varepsilon $-approximate solution of the problem. If $\varepsilon \leq 0$, then the solution corresponds to the exact case. Let $(\mathbf{\tilde{x}}, \mathbf{\tilde{y}})$ denotes the solution to problem $(\mathcal{P}^{\bar{\eta}}(\mathbf{b}, q))$, that is, we allow transactions provided that $\bar{\eta}_{i}>0$. Then,  
by \eqref{eq:general_bound}, we have 
\begin{equation}
 u(A(\mathbf{\tilde{x}}-\mathbf{\tilde{y}}))\geq u(A(\mathbf{\bar{x}}-\mathbf{\bar{y}})) \geq u(A(\mathbf{\tilde{x}} - \mathbf{\tilde{y}}))-\varepsilon (q, \bar{\eta})\ ,  \label{eq:prescribed}
\end{equation}%
where  
\begin{equation}
 \varepsilon (q,\bar{\eta}) := q_{\max }(\Vert \bar{\eta}\Vert _{0} - \Vert \bar{\eta} \Vert_{1}) + (q_{\max }-q_{\min })\Vert \bar{\eta}\Vert _{1}.
\label{def:epseta}
\end{equation}
 
As $\Vert \bar{\eta}\Vert _{\infty }\leq 1$, $\Vert \bar{\eta}\Vert_{0} - \Vert \bar{\eta} \Vert _{1}\geq 0$ since each $|\bar{\eta}_{i}|\leq 1$  
implies $\Vert \bar{\eta}\Vert _{1}\leq \Vert \bar{\eta}\Vert _{0}$, therefore $\varepsilon (q,\bar{\eta})\geq 0$ and $(\mathbf{\bar{x%
}},\mathbf{\bar{y}})$ constitutes an $\varepsilon (q,\bar{\eta})$-approximate solution to $(\mathcal{P}^{\bar{\eta}} (\mathbf{b}, q))$. Depending on the magnitude of $\varepsilon (q, \bar{\eta})$, this approximate solution may be acceptable for problem $(\mathcal{P}^{\bar{\eta}}(\mathbf{b}, q))$. 

We formalize this observation in the following result.
\end{remark}
 
 \begin{thr} \label{theo:value} 
 Let $(\mathbf{\breve{x}}, \mathbf{\breve{y}}, \breve{\eta}) \in \prod_{i=1}^{m}[0,R^{i}] \times \prod_{i=1}^{m}[0,b^{i}] \times \{0,1\}^{m}$ and $(\mathbf{\bar{x}}, \mathbf{\bar{y}}, \bar{\eta}) \in \prod_{i=1}^{m}[0,R^{i}]\times \prod_{i=1}^{m}[0,b^{i}] \times [0,1]^{m}$ be solutions to the problems $(\mathcal{P}^{r}(\mathbf{b},q))$ and 
$(\mathcal{P} (\mathbf{b},q))$, respectively. Let $(\mathbf{\tilde{x}}, \mathbf{\tilde{y}}, 1_{\left\{ \bar{\eta} > 0 \right\}}) \in \prod_{i=1}^{m}[0,R^{i}] \times \prod_{i=1}^{m}[0,b^{i}] \times \{0, 1\}^{m}$ where $(\mathbf{\tilde{x}},\mathbf{\tilde{y}})$ corresponds to solution of problem $(\mathcal{P}^{\bar{\eta}}(\mathbf{b},q))$.  Then, 
\begin{equation}\label{eq:ep-sol}
 \left\vert h(\mathbf{ \tilde{x}}, \mathbf{ \tilde{y}}, \bar{\eta}) - h(\mathbf{\breve{x}}, \mathbf{\breve{y}}, \breve{\eta}) \right\vert \leq \varepsilon (q, \bar{\eta}),
\end{equation}
where $h(\mathbf{x},\mathbf{y}, \eta ) := u(A(\mathbf{x} - \mathbf{y})) - \sum_{\eta _{i} > 0} q_{i}$, corresponds to the objective functional of problem $(\mathcal{P} (\mathbf{b}, q))$, that is, $(\mathbf{\tilde{x}}, \mathbf{\tilde{y}}, 1_{\left\{ \bar{\eta} > 0\right\}})$ is an $\varepsilon (q, \bar{\eta})$-approximate solution to problem $(\mathcal{P}(\mathbf{b}, q))$, where $\varepsilon (q,\bar{\eta}) = q_{\max} \left( \Vert \bar{\eta}\Vert_{0}-\Vert  \bar{\eta} \Vert _{1}\right) +\left( q_{\max} -q_{\min} \right) \Vert \bar{\eta} \Vert_{1}$. 
\end{thr}
\begin{proof}
Taking into account the optimality of $(\mathbf{\bar{x}}, \mathbf{\bar{y}}, \bar{\eta})$ and $(\mathbf{\breve{x}}, \mathbf{\breve{y}}, \breve{\eta})$, 
\begin{equation}
u(A(\mathbf{\bar{x}-\bar{y}}))-\sum\limits_{i=1}^{m}\bar{\eta}_{i}q_{i}\geq
u(A(\mathbf{\breve{x}}-\mathbf{\breve{y}}))-\sum\limits_{\breve{\eta}%
_{i}>0}q_{i}\geq u(A(\mathbf{\tilde{x}-\tilde{y}}))-\sum\limits_{\bar{\eta}%
_{i}>0}q_{i},  \label{eq:approx}
\end{equation}%
that we can rewrite 
\begin{equation*}
h^{r}(\mathbf{\bar{x}},\mathbf{\bar{y}},\bar{\eta})\geq h(\mathbf{\breve{x}},%
\mathbf{\breve{y}},\breve{\eta})\geq h(\mathbf{\tilde{x}},\mathbf{\tilde{y}},%
\bar{\eta}),
\end{equation*}%
where $h^{r}(\mathbf{\bar{x}},\mathbf{\bar{y}},\bar{\eta}):=u(A(\mathbf{\bar{%
x}-\bar{y}}))-\sum\limits_{i=1}^{m}\bar{\eta}_{i}q_{i}.$ From this, 
\begin{equation}
|h(\mathbf{\breve{x}},\mathbf{\breve{y}},\breve{\eta})-h(\mathbf{\tilde{x}},%
\mathbf{\tilde{y}},\bar{\eta})\rvert \leq \lvert h^{r}(\mathbf{\bar{x}},%
\mathbf{\bar{y}},\bar{\eta})-h(\mathbf{\tilde{x}},\mathbf{\tilde{y}},\bar{%
\eta})\rvert .  \label{280425}
\end{equation}%
Now, since $u(A(\mathbf{\bar{x}-\bar{y}}))\leq u(A(\mathbf{\tilde{x}-\tilde{y%
}}))$, we get  
\begin{eqnarray*}
\left\vert h^{r}(\mathbf{\bar{x}},\mathbf{\bar{y}},\bar{\eta})-h(\mathbf{%
\tilde{x}},\mathbf{\tilde{y}},\bar{\eta})\right\vert  &=&h^{r}(\mathbf{\bar{x%
}},\mathbf{\bar{y}},\bar{\eta})-h(\mathbf{\tilde{x}},\mathbf{\tilde{y}},\bar{%
\eta}) \\
&=&u(A(\mathbf{\bar{x}-\bar{y}}))-u(A(\mathbf{\tilde{x}-\tilde{y}}%
))+\sum\limits_{\bar{\eta}_{i}>0}q_{i}-\sum\limits_{\bar{\eta}_{i}>0}\bar{%
\eta}_{i}q_{i} \\
&\leq &\sum\limits_{\bar{\eta}_{i}>0}(1-\bar{\eta}_{i})q_{i}
\end{eqnarray*}%
From Remark \ref{rem:comparison}, $\sum\limits_{\bar{\eta}_{i}>0}(1-%
\bar{\eta}_{i})q_{i}\leq \varepsilon (q,\bar{\eta})$, using \eqref{280425}
we finally get%
\begin{equation*}
 \left\vert h^{r}(\mathbf{\bar{x}}, \mathbf{ \bar{y}}, \bar{\eta}) - h(\mathbf{\tilde{x}}, \mathbf{\tilde{y}},\bar{\eta}) \right\vert \leq \varepsilon (q, \bar{\eta}).
\end{equation*}
\end{proof}

We finish this section with the following remark regarding the technical assumption of pseudoconvexity for the composition function $f = - u \circ A$.

\begin{remark}
In Theorems \ref{theo:suff:KKT} and \ref{charact:KKT}, we are assuming that
the function $f(\mathbf{x},\mathbf{y},\eta )=-u\circ A(\mathbf{x}-\mathbf{y}%
)+\langle q,\eta \rangle $ is pseudoconvex. This assumption is not
restrictive as we analyze below. Indeed, simplifying notation, if $v:\mathbb{R%
}^{m}\rightarrow $ $\mathbb{R}$ is pseudoconvex ($v\equiv -u$) and $A:%
\mathbb{R}^{n}\rightarrow \mathbb{R}^{m}$ is a linear operator, then the
function $h:=v\circ A$ is pseudoconvex. Let $x_{1},x_{2}\in \mathbb{R}^{m}$
be such that $h(x_{1})<h(x_{2})$. Then, 
\begin{equation*}
\begin{array}{ll}
 h(x_{1})<h(x_{2}) & \Longleftrightarrow \, v(Ax_{1})<v(Ax_{2}) \\ 
 & \Longrightarrow \, \nabla v(Ax_{2})(Ax_{1}-Ax_{2})<0\text{ (}v\text{ pseudoconvex)} \\ 
 & \Longrightarrow \, \nabla v(Ax_{2})A(x_{1}-x_{2})<0 \\ 
 & \Longleftrightarrow \, \nabla h(x_{2})(x_{1}-x_{2}) < 0,
\end{array}
\end{equation*}
and $h = v \circ A$ is pseudoconvex.
\end{remark}

\section{No-trade Characterization with Gas Fees}\label{sec:04}

A no-trade condition prevents arbitrage by ensuring that for every solution $(\mathbf{\bar{x}}, \mathbf{\bar{y}}, \bar{\eta}) \in \mathbf{\hat{S}}$ to problem \eqref{GF}, we have $(\mathbf{\bar{x}}, \mathbf{\bar{y}}) = (\mathbf{0}, \mathbf{0})$, which implies that $\bar{\eta}=0$ by Remark \ref{rem:mu}. As a consequence, we have the following definition of no-trade for the optimal
routing problem.

\begin{defn}\label{def:notrade}
 It is said that a no-trade condition is verified in the optimal routing problem when $(\mathbf{\bar{x}}, \mathbf{\bar{y}, \bar{\eta}}) = (\mathbf{0}, \mathbf{0}, 0)$ is the unique solution of \eqref{GF}. 
\end{defn} 

Since $(\mathbf{0},\mathbf{0},0)$ is a feasible solution for problem %
\eqref{rP}, it is indeed a valid solution if the no-trade condition holds.
Therefore, we can apply our earlier analysis to problem \eqref{rP} in order
to derive properties under the assumption that the no-trade condition is
satisfied.

\begin{thr}\label{theo:notrade}
 Assume properties
 \eqref{p:Rb}, \eqref{p:utility}, and \eqref{p:market} are verified.  Consider the fo\-llo\-wing statements:
\begin{itemize}
 \item[$(a)$] The no-trade condition is satisfied. 

 \item[$(b)$] $\bar{P}^{i} \in K_{i}^{\gamma_{i}, q_{i}, b^{i}}$ for every $i \in \{1,\ldots,m\}$, where 
\begin{equation}
 K_{i}^{\gamma_{i},q_{i},b^{i}}:=\left\{ 
\begin{array}{l}
 z \in \mathbb{R}^{n}:\,\alpha _{i}z\geq (A^{i})^{\top }\nabla u(\mathbf{0})^{\top }\geq \gamma _{i}\alpha _{i}z-\mu^{i}, \\ 
 \text{for some }(\alpha _{i}, \mu^{i}) \in \mathbb{R}_{+} \times         \mathbb{R}_{+}^{n}\text{ and }q_{i}\geq (\mu^{i})^{\top} b^{i}  
\end{array}
\right\}. \label{new:set}
\end{equation}
\end{itemize}
Then, $(a) \Rightarrow (b)$. If, in addition, $f$ is pseudoconcave and $\varphi_{i}$ is quasilinear for every $i = 1,\ldots,m$, then $(b) \Rightarrow (a)$.
\end{thr}

\begin{proof}
$(a) \Rightarrow (b)$: we apply Theorem \ref{nec:KKT} (the necessary KKT conditions). The point $(\mathbf{\bar{x}}, \mathbf{\bar{y}}, \bar{\eta}) = (\mathbf{0}, \mathbf{0},0)$ trivially satisfies \eqref{def:comple} and \eqref{p:y}, as $b^{i} >> 0$ by con\-di\-tion \eqref{p:Rb}. In this case, $\bar{\eta}_{i} = 0$ for every $i=1,\ldots,m$.

For a generic $i \in \{1,\ldots,m\}$, the following condition holds  
\begin{equation}
\begin{cases}
 \alpha_{i} (\bar{P}^{i})^{\top} \geq \nabla u(\mathbf{0}) A^{i} \geq \gamma_{i} \alpha_{i} (\bar{P}^{i})^{\top} - (\mu^{i})^{\top}, \\ 
 q_{i} - (\mu^{i})^{\top} b^{i} \geq 0. \label{no:trade}
\end{cases}
\end{equation}
and this condition is equivalent to \eqref{no:trade}.

\noindent $(b) \Rightarrow (a)$: Directly by applying Theorem \ref{theo:suff:KKT} (sufficient KKT conditions). 
\end{proof}

\begin{remark} \label{rem:gasfee}
A geometric interpretation of Theorem \ref{theo:notrade}$(b)$ is that the
set in \eqref{new:set} implies that the prices belong to: 
\begin{equation}
P^{i}\in K_{i}^{\gamma _{i}}+\prod_{j=1}^{n}\left[ -\frac{q_{i}}{b_{j}^{i}},0%
\right] ,~\forall ~i\in \{1,\ldots ,m\},  \label{reg:cone}
\end{equation}%
where $K_{i}^{\gamma _{i}}:=\{z\in \mathbb{R}^{n}:\,\gamma _{i}\nu _{i}z\leq
A_{i}\nabla u(0)\leq \nu _{i}z,\text{ for some }\nu _{i}\in \mathbb{R}\}$. Hence, each price $P^{i}$ belongs to the set $K_{i}^{\gamma _{i}}$ perturbed by an interval that depends on the fixed costs $q_{i}$ of the markets and the trade limits $b_{j}^{i}$ for all $j\in \{1,\ldots,n_{i}\}$.

This can be deduced from the second inequality in \eqref{new:set}. Indeed, from \eqref{new:set} we derive that 
\begin{equation*}
-\sum_{k=1}^{n}\mu _{k}^{i}b_{k}^{i}\geq -q_{i}\Longrightarrow -\mu
_{j}^{i}\geq -\frac{q_{i}}{b_{j}^{i}}+\sum_{k\neq j}\mu _{k}^{i}\frac{%
b_{k}^{i}}{b_{j}^{i}}\geq -\frac{q_{i}}{b_{j}^{i}},~\forall ~j=1,\ldots
,n_{i}.
\end{equation*}%
Thus,
\begin{align*}
&\alpha_{i} (\bar{P}^{i})^{\top }\geq \nabla u(0)A^{i} \geq \gamma_{i} \alpha_{i}(\bar{P}^{i})^{\top} - (\mu^{i})^{\top} \\ \overset{\eqref{no:trade}}{\implies} &\alpha_{i} \bar{P}_{j}^{i} \geq (\nabla u(0)A^{i})_{j} \geq \gamma_{i} \alpha_{i} \bar{P}_{j}^{i}-\frac{q_{i}}{b_{j}^{i}}.
\end{align*}
Then, \eqref{reg:cone} holds.
\end{remark}

\begin{remark} \label{rem:comp}
To the best of our knowledge, the no-trade characterization for the optimal routing model with gas fees is a novel contribution to the literature, even in the single-market case. Previous works have derived no-trade conditions under different assumptions:
\begin{itemize}
 \item[$(i)$] In \cite[Section~3]{OptRou-2022}, the no-trade condition is established for a convex relaxation of the optimal routing problem without gas fees, so the market conditions were simpler.

 \item[$(ii)$] For the single-market case, \cite{angeris2022constant} derives the no-trade condition for the relaxed convex problem, while \cite{ELS-1} extends this result to the original non-convex formulation.
\end{itemize}
In all cases, our model recovers the no-trade condition when the gas fee is zero, thereby extending the previous results. This also aligns with Remarks \ref{rem:yi} and \ref{rem:gasfee}, confirming that gas fees introduce an additional term in the no-trade region, consequently expanding it and preventing arbitrage opportunities.
\end{remark}

\section{Interpretations and Numerical Experiments}\label{sec:05}

In this section, we describe two numerical experiments we have performed in order to exemplify our theoretical results. Their outcomes facilitate the financial interpretation of the statements in the previous sections.

\begin{ex} 
For our first example, we study the influence of the gas fee in the no-trade
region. For that we consider a simple model with a single market $m=1$, six
tokens $n=6$, and a linear utility $u(z)=\pi ^{T}z$ based on the one proposed
in \cite[Section 5.3]{angeris2022constant}, and studied in \cite[Section 5.2]%
{ELS-1}. In the present context, we can write the problem as a particular
instance of problem \eqref{rGF}, but we omit superscripts for notational
simplification, so it reads: 
\begin{equation}
\begin{array}{ll}
\text{maximize} & \pi ^{T}(x-y)-q\eta \\ 
\text{subject to} & y-\eta b\leq 0,\  \\ 
& \varphi (R+\gamma {y}-{x})-\varphi (R)=0, \\ 
& 0\leq x\leq R,\text{ }0\leq {y^{i}},\text{ }0\leq \eta \leq 1.%
\end{array}
\label{210425}
\end{equation}%
With respect to the model in \cite{angeris2022constant}, three news elements
are considered herein: the activation variable $\eta \in \mathbb{R}$, the
gas fee $q\in \mathbb{R}$, and the tender basket bound $b\in \mathbb{R}^{6}$.
We consider the same data as in the references, a discount rate $\gamma =0.9$
and a vector of reserves 
\begin{equation*}
R=(1,3,2,5,7,6).
\end{equation*}%
Following Remark \ref{rem:KRZ}, we take $b=2 R/\gamma$ which is
consistent with the theory. The vector $\pi \in \mathbb{R}^{6}$ models the
trader private prices and, following the example given in \cite{ELS-1}, we assume
that these prices are parametrized by

\begin{equation*}
\pi =\pi (t,s)\equiv (tp_{1},sp_{2},...,p_{6})\text{ for some scalars }%
t,s\in \left[ 0,2\right] ,
\end{equation*}%
where the vector of scaled market prices $\ p\in \mathbb{R}^{6}$ \ is given by its components $p_{i}=P_{i}/P_{6}=\nabla \varphi (R)_{i}/\nabla \varphi (R)_{6}$. In this way, the market and trader prices coincide for $s=t=1$, so it is expected that the no-trade region is located around this value. We
compare two types of market functions. On one hand, the geometric mean 
\begin{equation*}
\varphi _{\mathrm{gm}}(x)=\prod\limits_{i=1}^{n}x_{i}^{1/n}
\end{equation*}%
and, on the other, the weighted quasi-arithmetic mean introduced in \cite
{ELS-1}, 
\begin{equation*}
\varphi _{\mathrm{qm}}(x)=\exp \left( W_{0}\left\{ \frac{2}{n}\left[
\sum_{i=1}^{n}(x_{i}+1)^{2}\ln (x_{i}+1)\right] \right\} \bigg/2\ \right) -1,
\end{equation*}%
where $W_{0}(\cdot)$ is the principal branch of the Lambert omega function. The motivation to introduce this market function is described in \cite{ELS-1}, and is connected to the fact that it is neither convex nor concave.

In both cases, all the hypotheses for the equivalence between {\it (a)} and {\it (b)} in the statement of Theorem \ref{theo:notrade} are verified, since $u$ is increasing and linear, thus pseudoconvex, and $\varphi _{\mathrm{%
gm}}$ and $\varphi _{\mathrm{qm}}$ are increasing and quasilinear, see \cite{ELS-1}. We numerically compute the optimal trades by solving the corresponding problem \eqref
{210425} for a equispaced grid of points partitioning $\left[ 0,2\right] \times \left[
0,2\right] $ and different values of the gas fee; then we identify when the
no-trade condition is verified. Before showing our results, we note that
some parameter values might be unrealistic in practice, as we have selected them for the sake of illustration. Numerical problems are
solved by means of the SciPy optimization library \cite{Virtanen}.

The numerical results in Figures \ref{fig:Ex1_gm_A} and \ref{fig:Ex1_qm_A} confirm that gas fees increase the size of the no-trade region in both cases. Comparing the market functions reveals that the weighted quasi-arithmetic mean generates a larger no-trade region for this particular example. In terms of the market activation function $\eta(t,s)$, the geometric mean produces a smoother profile, while the quasi-arithmetic mean exhibits a piecewise constant, nonsmooth pattern. This difference leads to more frequent market activations ($\eta(t,s) = 1$) in the neighborhood of the no-trade region boundary, as can be seen by comparing Figures \ref{fig:Ex1_gm_B} and \ref{fig:Ex1_qm_B}.
\end{ex}

\begin{ex}
In the second example, we consider an optimal routing problem with $m=5$
markets and $n=3$ tokens. In this case, we analyze the linear utility model studied in \cite[Section 4]{OptRou-2022}; in particular, we consider
the same network topology, see Figure \ref{fig:Ex2_topology}, where the markets (also the market functions and reserves) follow the Roman numbering system. The discount rates are taken to be $\gamma _{i}=0.99$, a small value in practice, but we have selected it for the sake of illustrating the no-trade condition. See Table \ref{tab:Ex2} for these as well as the rest of data. We consider
the specific linear utility problem:
\begin{equation}
\begin{array}{ll}
\text{maximize} & \pi^{T} \left( {\sum\limits_{i=1}^{5} A^{i}}{(x^{i}-y^{i})} \right) - \sum\limits_{i=1}^{5}q_{i} \eta _{i} \\ 
\text{subject to} & y^{i}\leq \eta _{i}b^{i}, \\ 
& \varphi _{i}(R^{i}+\gamma _{i}{y^{i}}-{x^{i}})=\varphi _{i}(R^{i}), \\ 
& 0\leq x^{i}\leq R^{i},\text{ }0\leq {y^{i}}, \\ 
& 0\leq \eta _{i}\leq 1,~\forall ~i\in \{1,\ldots ,5\}.\ 
\end{array}
\label{220426}
\end{equation}%
As in the previous example, we take $b^{i} = 2 R^{i}/\gamma_{i}$, which is consistent with the theory by Remark \ref{rem:KRZ}. To carry out a similar experiment as before, we consider the following vectors of local prices corresponding to each market:
\begin{align*}
 & P^{1}=(0.16881825,1.68818239,0.16881823), \\
 & P^{2}=(0.15811388,1.58113882), \\ 
 & P^{3}=(1.58113882,0.15811388), \\ 
 & P^{4}=(0.79056903,0.3162274), \\
 & P^{5}=(1,1),
\end{align*}
which have been numerically calculated. We model the trader private prices for the three different tokens as a uniparametric perturbation of the local price in Market 1, in particular
\begin{equation*}
\pi =(tP_{1}^{1},P_{2}^{1},P_{3}^{1})\text{ for some }t\in [0.2,9].
\end{equation*}
The gas fee is fixed at $q_{i} = 0.01 $. As previously, the assumptions of Theorem \ref{theo:notrade} are fulfilled, and therefore we can apply it to find out the conditions under which there is no trade. Specifically, setting $t = 1$ and using the geometric approach given by equation \eqref{reg:cone}, the no-trade condition implies that there exists a vector $(\alpha_{1}, \alpha_{2}, \alpha_{3}, \alpha_{4}, \alpha_{5}) \in \mathbb{R}_{+}^{5}$ such that the following inequalities hold for each market:

\textsc{Market 1}

\begin{equation*}
\begin{cases}
\alpha_{1} \gamma_{1} P_{1}^{1} - \frac{q_{1}}{b_{1}^{1}} \leq P_{1}^{1}
\leq \alpha_{1} P_{1}^{1}, \\ 
\alpha_{1} \gamma_{1} P_{2}^{1} - \frac{q_{2}}{b_{2}^{1}} \leq P_{2}^{1}
\leq \alpha_{1} P_{2}^{1}, \\ 
\alpha_{1} \gamma_{1} P_{3}^{1} - \frac{q_{3}}{b_{3}^{1}} \leq P_{3}^{1}
\leq \alpha_{1} P_{3}^{1}.%
\end{cases}%
\end{equation*}

\textsc{Market 2}

\begin{equation*}
\begin{cases}
\alpha_{2} \gamma_{2} P_{1}^{2} - \frac{q_{2}}{b_{1}^{2}} \leq P_{1}^{1}
\leq \alpha_{2} P_{1}^{2}, \\ 
\alpha_{2} \gamma_{2} P_{2}^{2} - \frac{q_{2}}{b_{2}^{2}} \leq P_{2}^{1}
\leq \alpha_{2} P_{2}^{2}.%
\end{cases}%
\end{equation*}

\textsc{Market 3}

\begin{equation*}
\begin{cases}
\alpha_{3} \gamma_{3} P_{1}^{3} - \frac{q_{3}}{b_{1}^{3}} \leq P_{2}^{1}
\leq \alpha_{3} P_{1}^{3}, \\ 
\alpha_{3} \gamma_{3} P_{2}^{3} - \frac{q_{3}}{b_{2}^{3}} \leq P_{3}^{1}
\leq \alpha_{3} P_{2}^{3}.%
\end{cases}%
\end{equation*}

\textsc{Market 4}

\begin{equation*}
\begin{cases}
\alpha_{4} \gamma_{4} P_{1}^{4} - \frac{q_{4}}{b_{1}^{4}} \leq P_{1}^{1}
\leq \alpha_{4} P_{1}^{4}, \\ 
\alpha_{4} \gamma_{4} P_{2}^{4} - \frac{q_{4}}{b_{2}^{4}} \leq P_{3}^{1}
\leq \alpha_{4} P_{2}^{4}.%
\end{cases}%
\end{equation*}

\textsc{Market 5}

\begin{equation*}
\begin{cases}
\alpha_{5} \gamma_{5} P_{1}^{5} - \frac{q_{5}}{b_{1}^{5}} \leq P_{1}^{1}
\leq \alpha_{5} P_{1}^{5}, \\ 
\alpha_{5} \gamma_{5} P_{2}^{5} - \frac{q_{5}}{b_{2}^{5}} \leq P_{3}^{1}
\leq \alpha_{5} P_{2}^{5}.%
\end{cases}%
\end{equation*}

Through direct computation, we verify that the parameter values $\alpha_{1} = 1$, $\alpha_{2} = \alpha_{3} = \frac{P_{1}^{2}}{P_{2}^{2}} = \frac{P_{1}^{2}}{P_{1}^{3}}\approx 1.0677$, $\alpha _{5}=P_{1}^{1}$ satisfy the no-trade
conditions for Markets 1, 2, 3, and 5, while Market 4 presents an unsolvable condition. The second inequality requires: 
\begin{equation}
\alpha _{4}\geq \max \left\{ \frac{P_{1}^{1}}{P_{1}^{4}},\frac{P_{3}^{1}}{%
P_{2}^{4}}\right\} \approx 0.53385,
\end{equation}%
which contradicts the first equation. Numerical verification confirms this
incompatibility: 
\begin{equation}\label{eq:bound} 
\alpha _{4}\gamma _{4}P_{1}^{4}-\frac{q_{4}}{b_{4}^{1}}\geq 0.53385\times
0.7115-2.25\times 10^{-4}>0.1688=P_{3}^{1}.  
\end{equation}

Consequently, in a neighborhood of $t=1$, we observe no trading activity in Markets 1, 2, 3, and 5, while arbitrage opportunities in Market 4. Numerical experiments with a 200-point grid confirms this behavior, see Figure~\ref{fig:Ex2_trades_q001_B}. Furthermore Market 5 remains inactive for $t\in \lbrack 0.2,0.9].$
This can also be checked from market activations, see Figure \ref%
{fig:Ex2_activation_markets_q001}, where $\bar{\eta}_{3}(t)\approx 0.25$ for 
$t\in \lbrack 0.2,0.9]$. We can check that no market is purely active, $\bar{%
\eta}_{i}(t)<1$, so $\bar{y}^{i}<b^{i}$ for every $i=1,\ldots,5$, and constraint qualification \eqref{KRZ} holds in the whole parameter range, while the optimality conditions are still valid. Consequently, the trade behavior follows by Remark \ref{rem:KRZ}.(i), as
numerically verified for Market 1 in Figure~\ref%
{fig:Ex2_Market1_OptimalTenderedBasket_q001}.  For Market 4, enforcing no-trade conditions requires either reducing the discount rate $\gamma _{4}$ or increasing the gas fee $q_{4}.$ In the second case, we can enforce the no-trade condition on Market 4 by taking 
\begin{equation}
 q_{4} \geq \max \left\{ \left( \frac{P_{3}^{1}}{P_{2}^{4}} \gamma_{4} P_{1}^{4} - P_{1}^{1} \right) b_{1}^{4},\left( \frac{P_{3}^{1}}{P_{2}^{4}} \gamma _{4}P_{2}^{4}-P_{3}^{1}\right) b_{2}^{4}\right\} =9.3788.
\end{equation}
Setting $q_{4}=9.4$ establishes a no-trade region, effectively deactivating Market 4, see Figure~\ref{fig:Ex2_trades_q4_9}. The numerical computation of the no-trade interval gives
\begin{equation}
\lbrack 0.907537688442211,1.0844221105527638],
\end{equation}
see Figure \ref{fig:Ex_2_Notrade_region_q4_94}. 

In a second experiment, we numerically validate Remark~\ref{rem:comparison}
and consequently Theorem~\ref{theo:value}. Given a solution $(\bar{\mathbf{x}%
}, \bar{\mathbf{y}}, \bar{\eta})$ to the relaxed problem $(\mathcal{P}^{r}(%
\mathbf{b},q))$, we compare $(\bar{\mathbf{x}}, \bar{\mathbf{y}})$ with the
solution $(\tilde{\mathbf{x}}, \tilde{\mathbf{y}})$ of the problem $(\mathcal{%
P}^{\bar{\eta}}(\mathbf{b},q))$, for which the activation of the market is determined by $%
\bar{\eta}$. We consider different constant gas fees $q_i \in \{0.01, 0.1,
0.5, 1\}$. Following Remark~\ref{rem:comparison}, the solution $(\bar{\mathbf{x}}, \bar{%
\mathbf{y}})$ is an $\varepsilon(q,\bar{\eta})$-approximate solution to $%
\mathcal{P}^{\bar{\eta}}(\mathbf{b},q)$, where 
\begin{equation}
\varepsilon(q,\bar{\eta}) := q\left(\|\bar{\eta}\|_0 - \|\bar{\eta}%
\|_1\right).
\end{equation}
This implies that the utility difference between both solutions should be
small and, by stability arguments, the trades should satisfy $(\bar{\mathbf{x%
}}, \bar{\mathbf{y}}) \approx (\tilde{\mathbf{x}}, \tilde{\mathbf{y}})$ when 
$q$ is sufficiently small for a moderate number of markets, as is the case in our
example. This result has practical significance because, by Theorem~\ref{theo:value}, 
$\left( \mathbf{\tilde{x}},\mathbf{\tilde{y}},1_{\left\{ \bar{\eta}>0\right\} }\right) $ is an $\varepsilon(q, 
\bar{\eta})$-approximate solution to the original problem $(\mathcal{P}(%
\mathbf{b},q))$, with $\varepsilon(q,\bar{\eta})$ being an a priori
computable bound. Our numerical computations, obtained under identical experimental conditions, confirm these theoretical observations. 
Figures \ref{fig:Ex2_utility_combined} and \ref{fig:Ex2_trades_comparison} demonstrate that for $q=0.01$, 
both utilities and trades are virtually identical. However, as the gas fees increase and $\varepsilon(q,\bar{\eta})$ 
becomes significantly larger, we observe noticeable divergences between these quantities. In this sense, in Figure \ref{fig:Ex2_utility_combined}, the staircase-like discontinuities in some of the plots are explained by the nature of the $\ell_0$ pseudonorm and the term $\sum\limits_{\bar{\eta}_i > 0} q_i$.
\end{ex}

\begin{figure}[htbp]
\centering
 \includegraphics[width=\linewidth,height=0.68\linewidth]{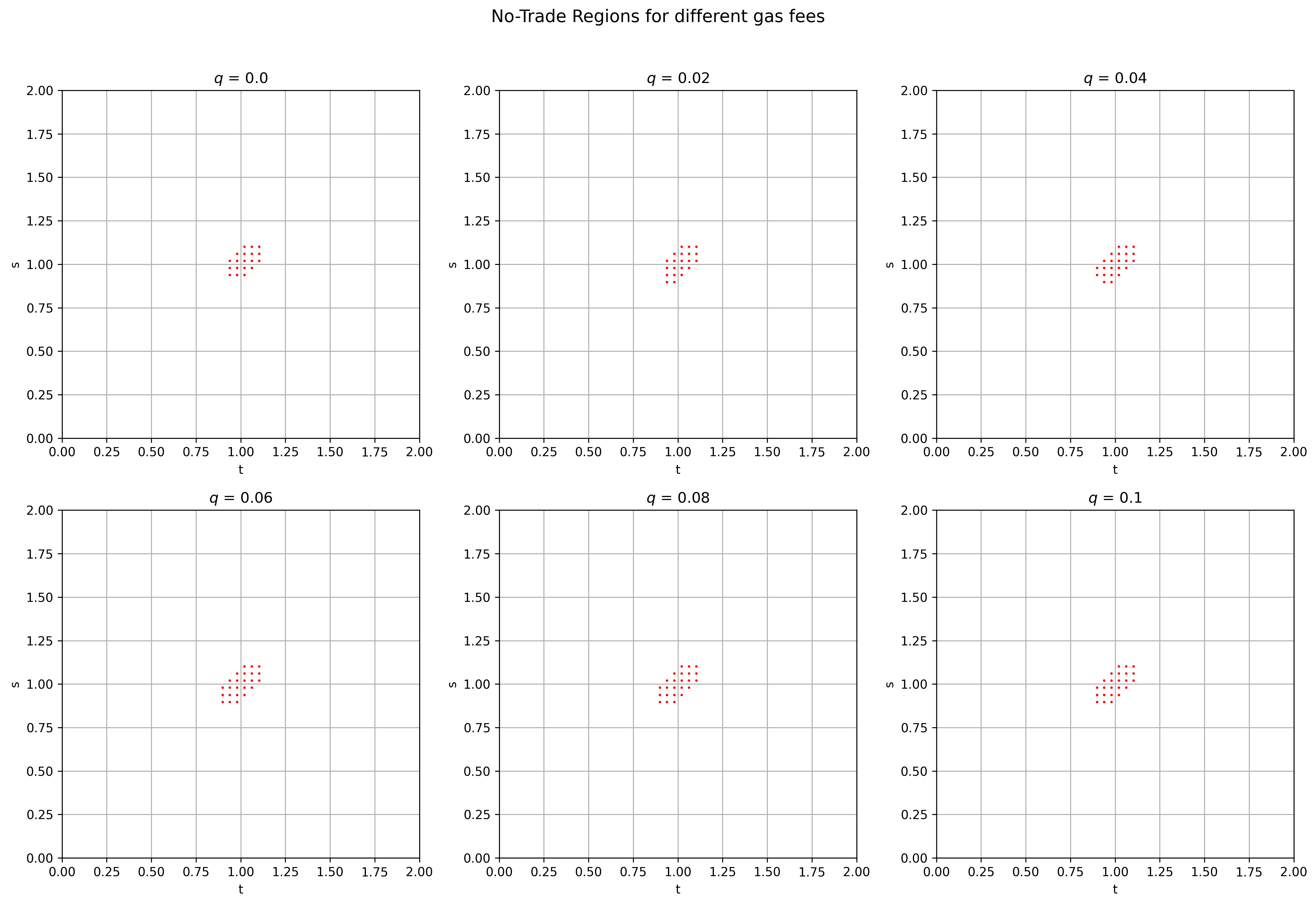}
\caption{Example 1: No-trade region for the geometric mean function $\varphi _{\mathrm{gm}}$.}
\label{fig:Ex1_gm_A}
\end{figure}

\begin{figure}[htbp]
\centering
 \includegraphics[width=\linewidth,height=0.68\linewidth]{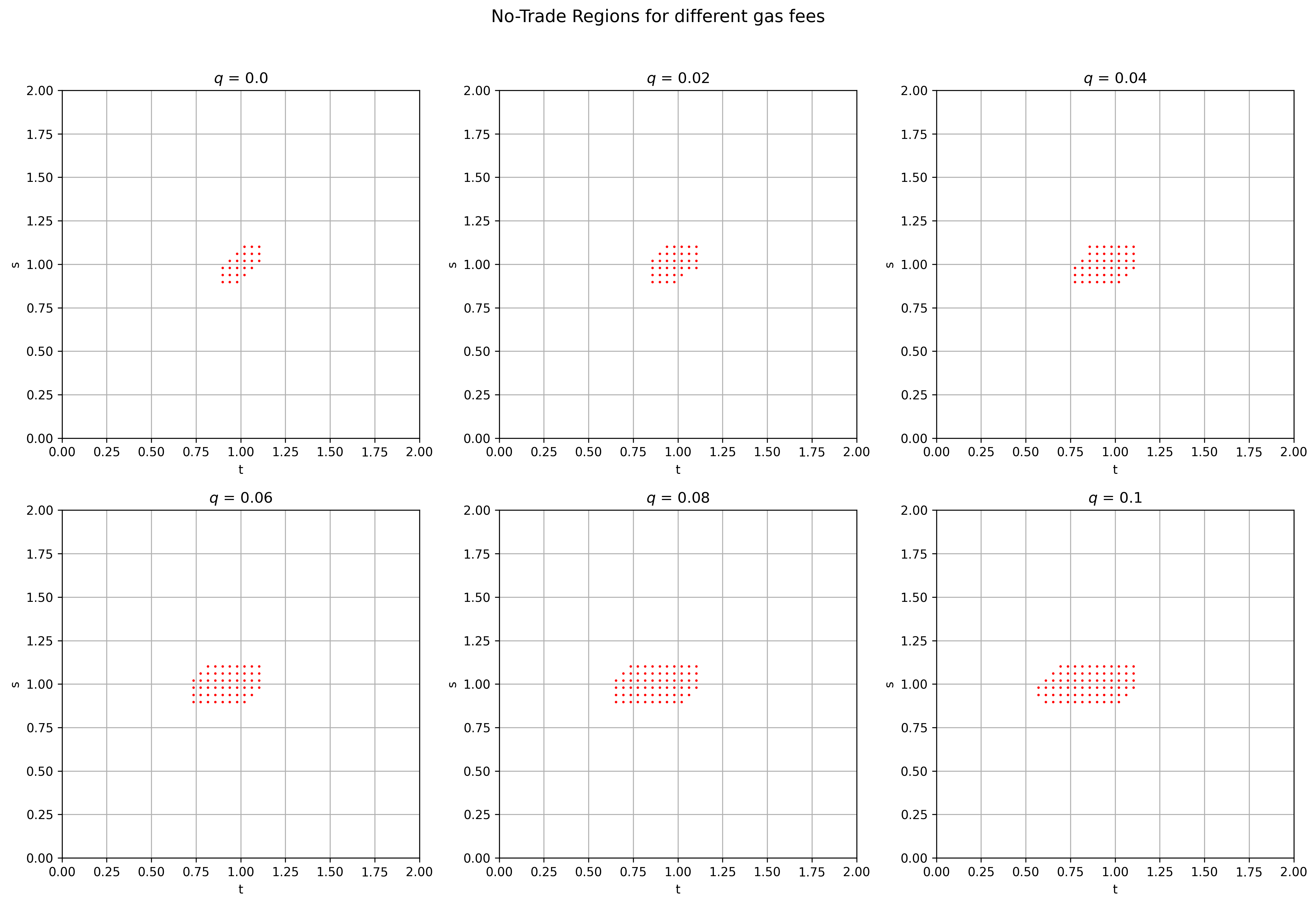}
\caption{Example 1: No-trade region for the weighted quasi-arithmetic mean function $\varphi _{\mathrm{qm}}$.}
\label{fig:Ex1_qm_A}
\end{figure}

\begin{figure}[htbp]
\centering
 \includegraphics[width=\linewidth,height=0.89\linewidth]{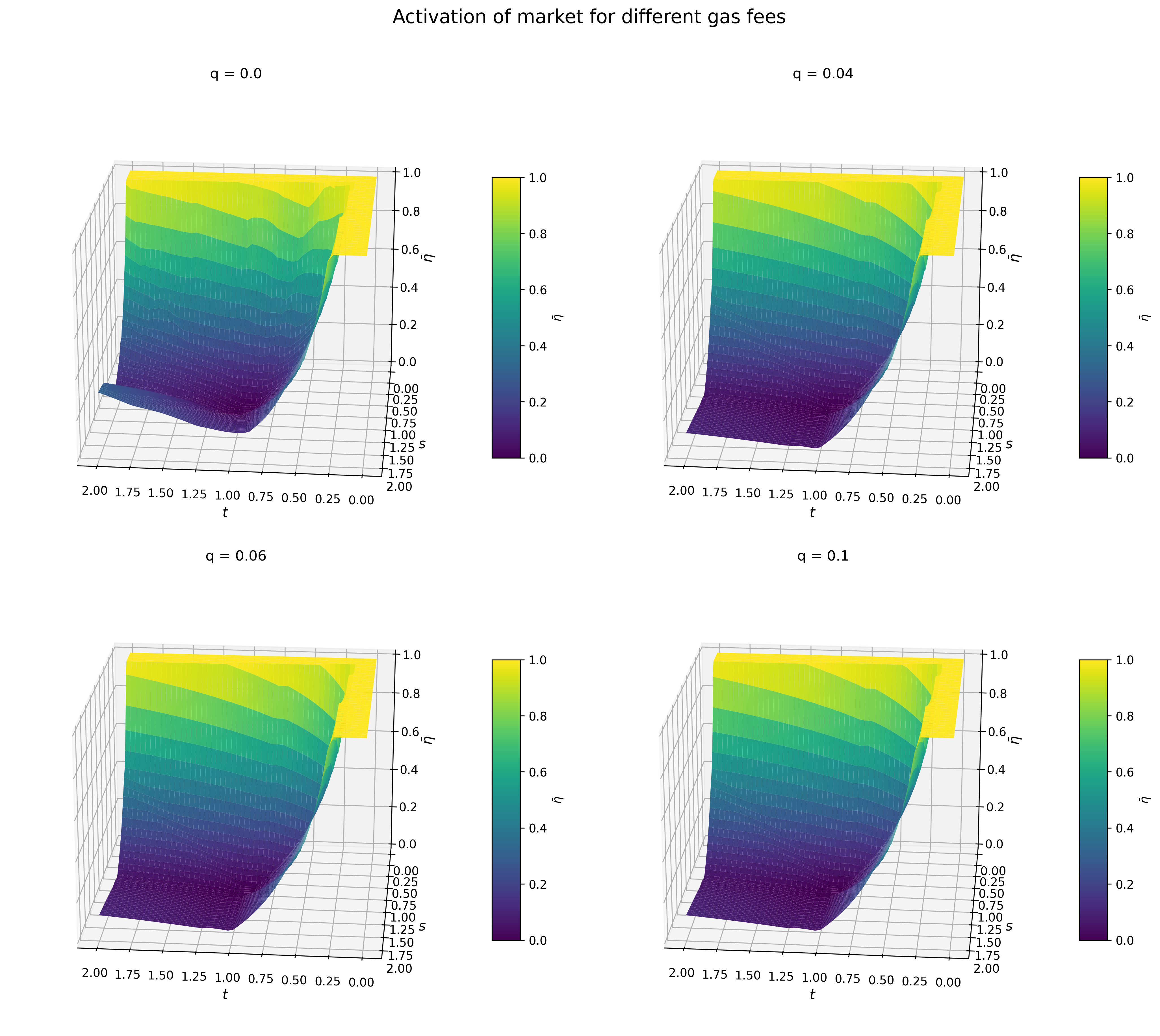}
\caption{Example 1: Activation of the market for $\varphi _{\mathrm{gm}}$.}
\label{fig:Ex1_gm_B}
\end{figure}

\begin{figure}[htbp]
\centering
 \includegraphics[width=\linewidth,height=0.89\linewidth]{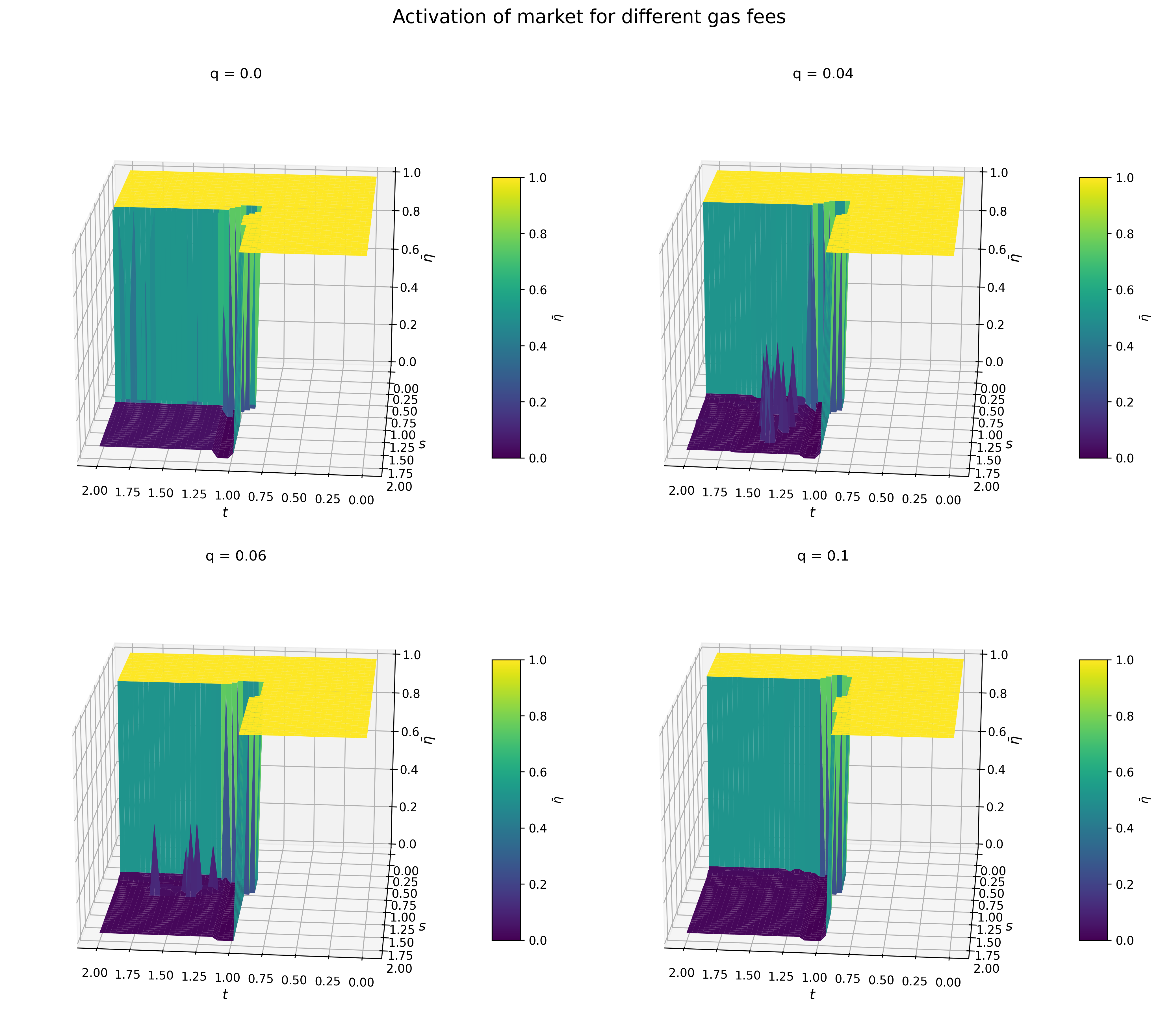}
\caption{Example 1: Activation of the market for $\varphi _{\mathrm{qm}}$.}
\label{fig:Ex1_qm_B}
\end{figure}

\begin{figure}[htbp]
\centering
 \includegraphics[width=0.8\linewidth,height=\linewidth]{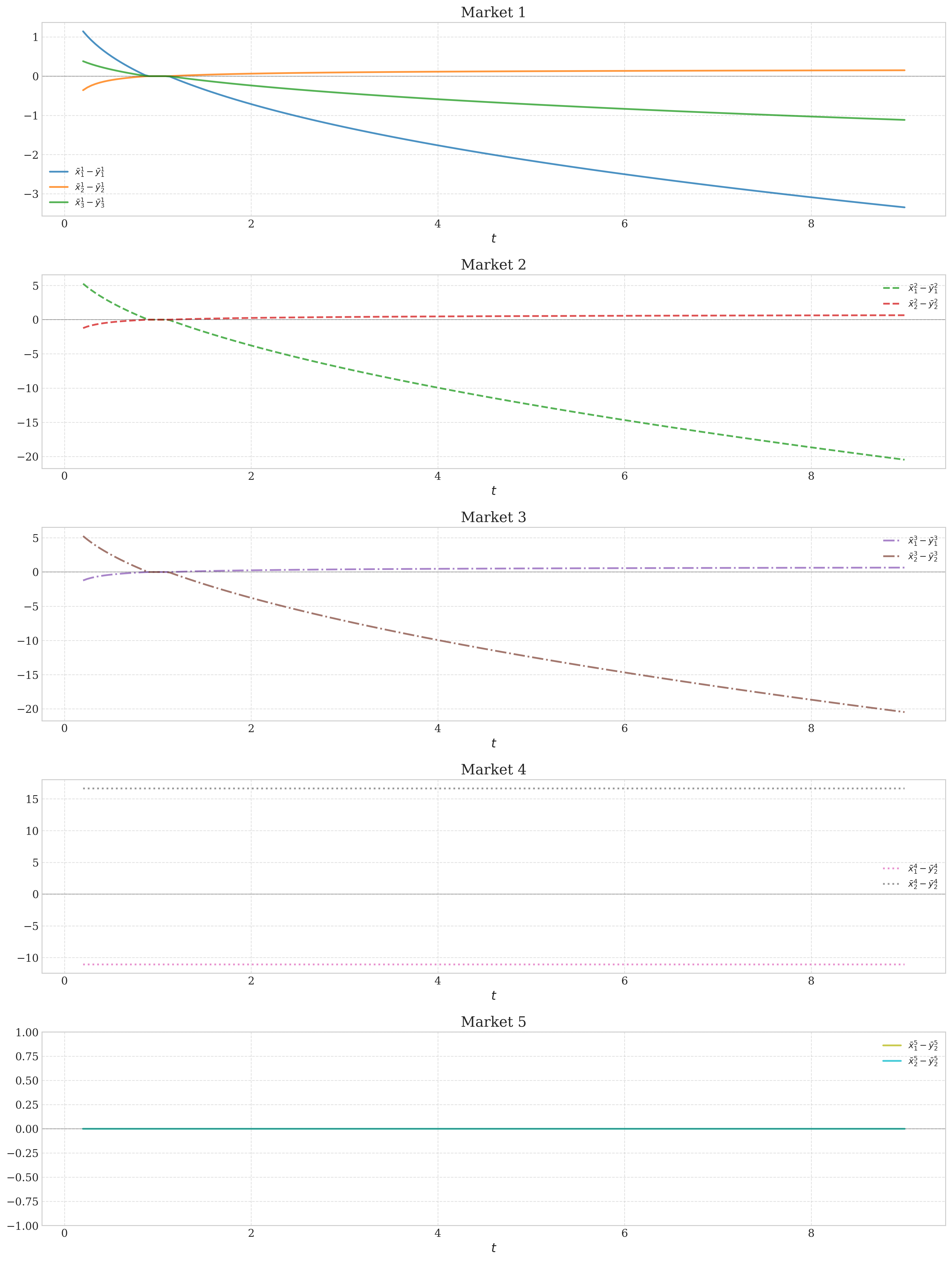}
\caption{Example 2: Optimal trades $(q_{i}=0.01)$.}
\label{fig:Ex2_trades_q001_B}
\end{figure}

\begin{figure}[htbp]
\centering
 \includegraphics[width=\linewidth,height=0.74\linewidth]{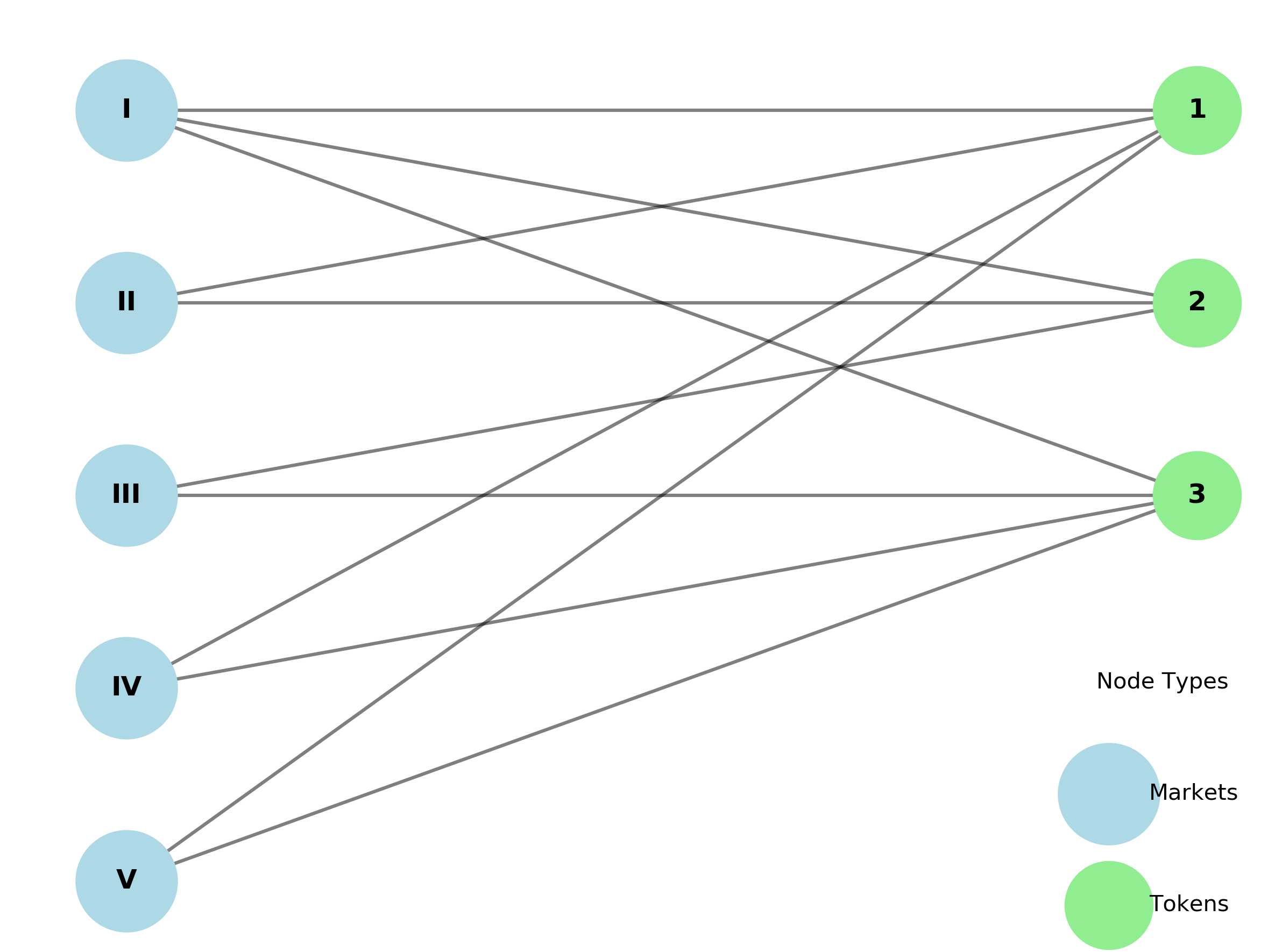}
\caption{Example 2: Market network topology.}
\label{fig:Ex2_topology}
\end{figure}

\begin{figure}[htbp]
\centering
 \includegraphics[width=\linewidth,height=0.6\linewidth]{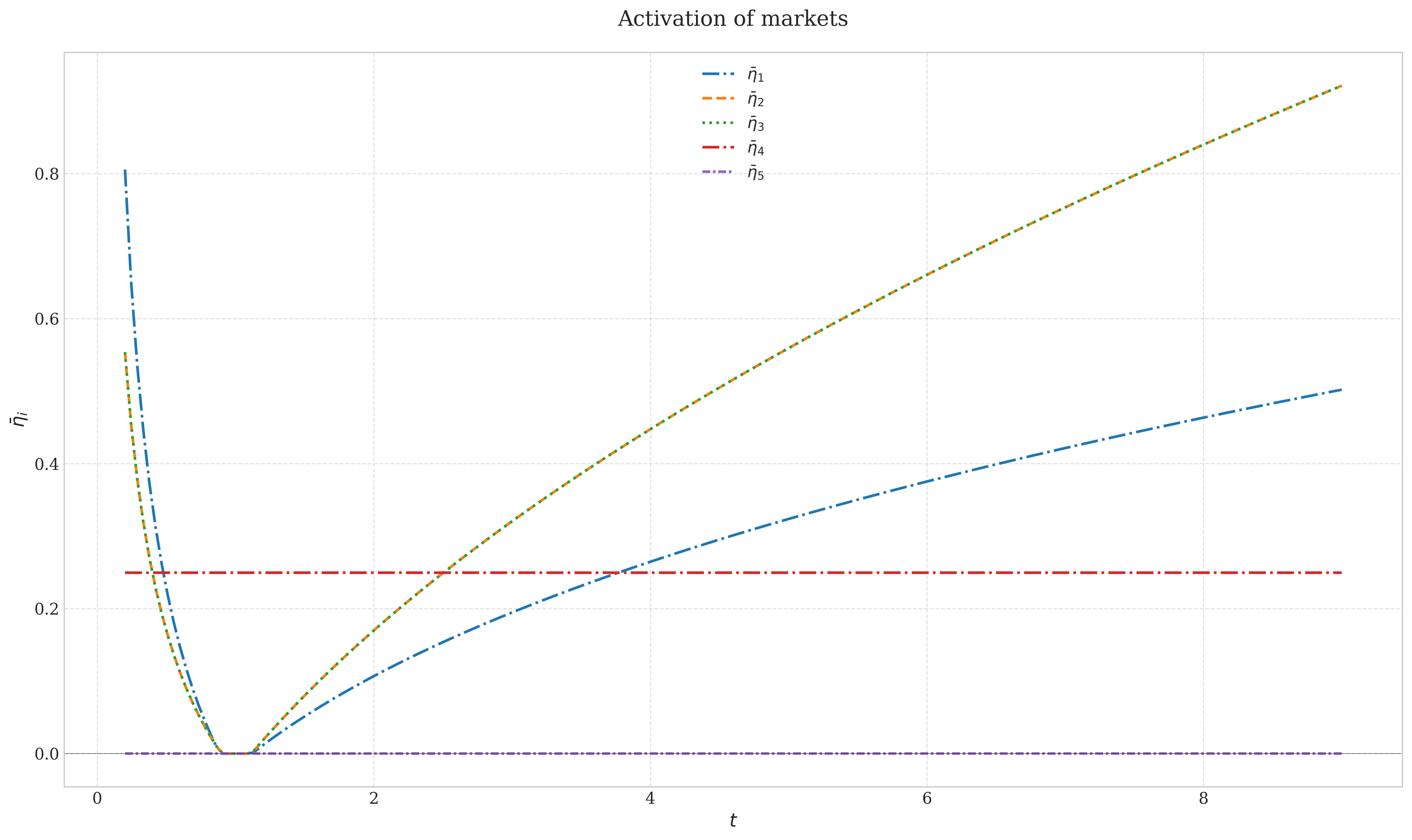}
\caption{Example 2: Activation of markets $(q_{i}=0.01)$.}
\label{fig:Ex2_activation_markets_q001}
\end{figure} 

\begin{figure}[htbp]
\centering
 \includegraphics[width=\linewidth,height=0.94\linewidth]{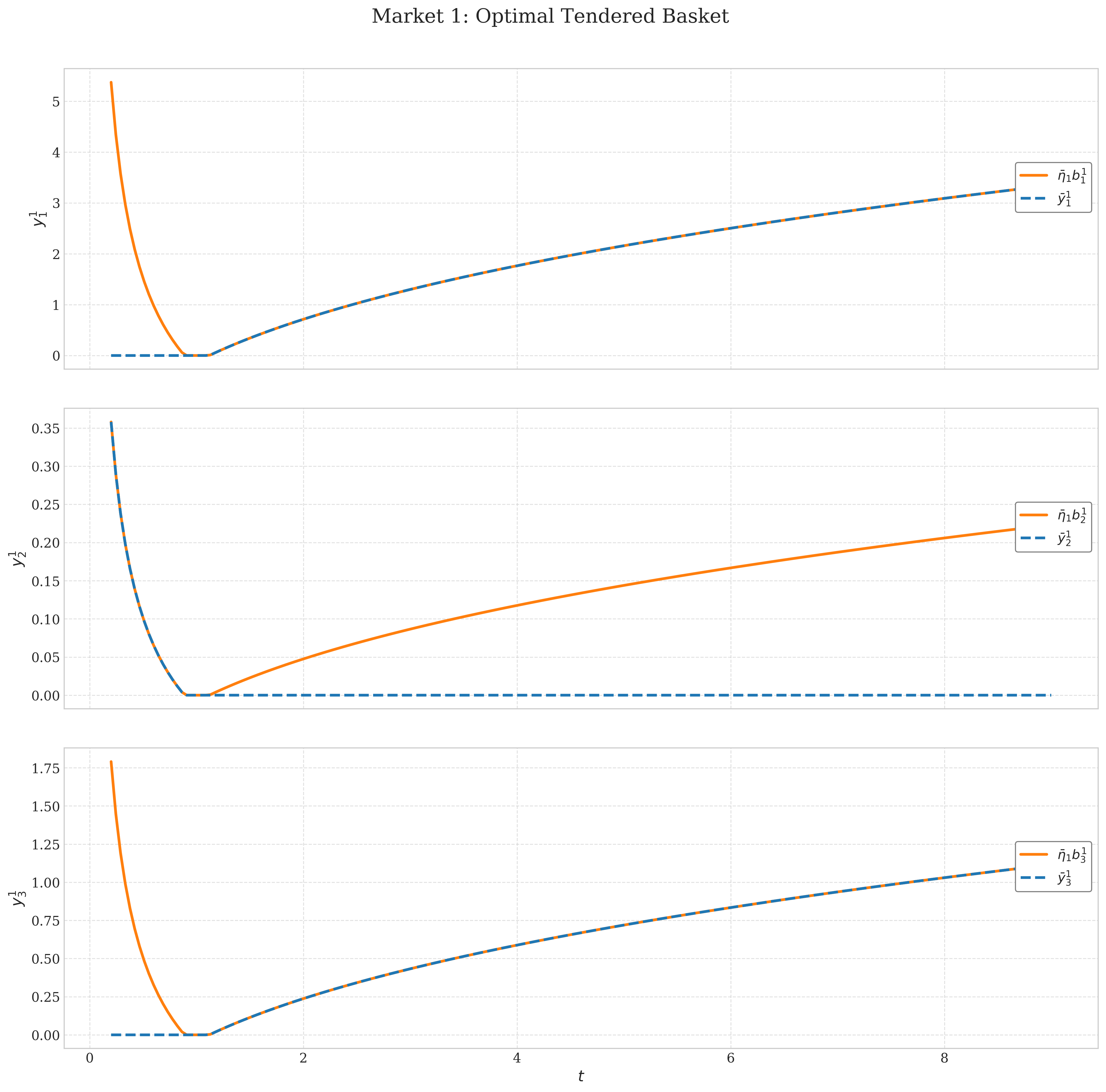}
\caption{Example 2: Optimal tendered basket in Market 1 $(q_{i}=0.01)$.}
\label{fig:Ex2_Market1_OptimalTenderedBasket_q001}
\end{figure}

\begin{figure}[htbp]
\centering
 \includegraphics[width=\linewidth,height=0.45\linewidth]{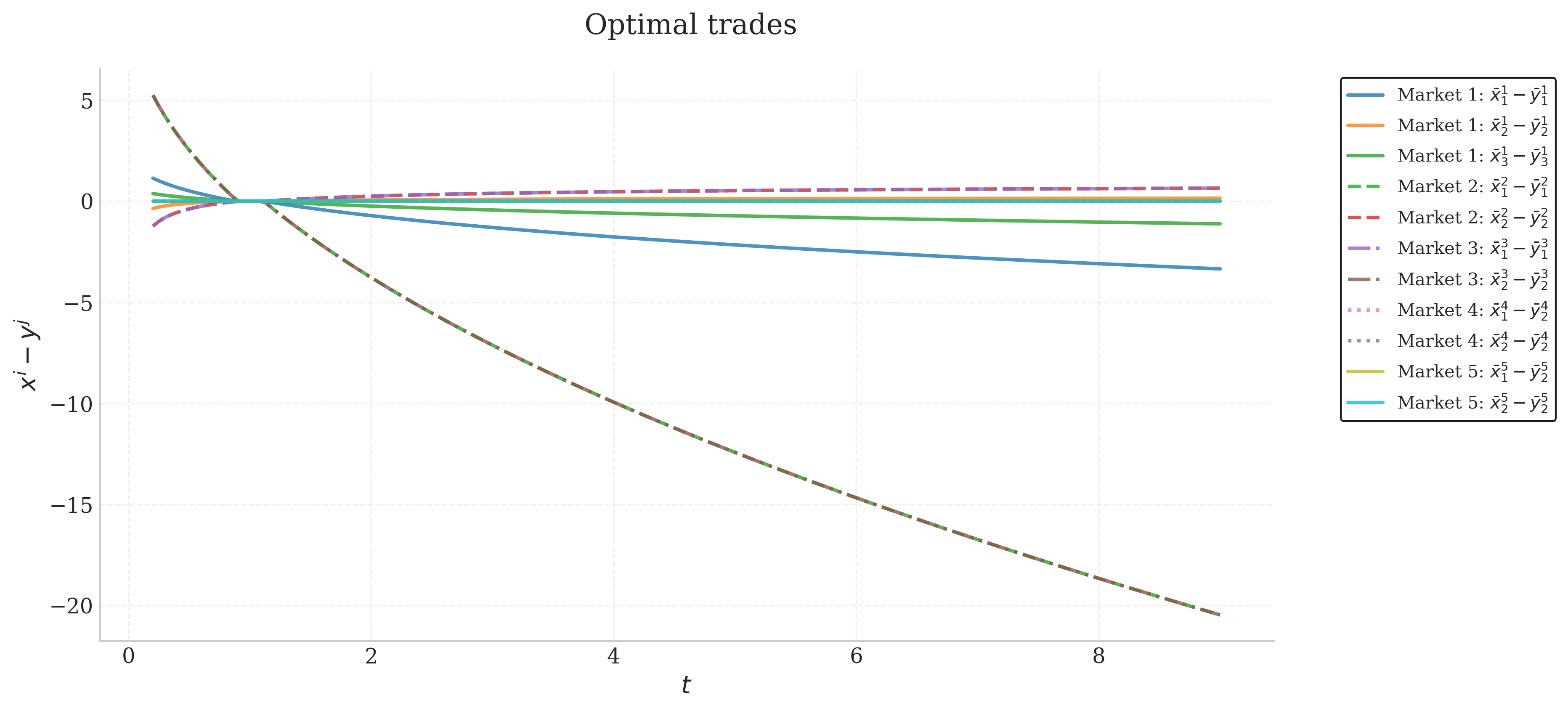}
\caption{Example 2: Optimal trades $(q_{i}=0.01,\ i\neq 4,q_{4}=9.4)$.}
\label{fig:Ex2_trades_q4_9}
\end{figure} 

\begin{figure}[htbp]
\centering
 \includegraphics[width=\linewidth,height=0.6\linewidth]{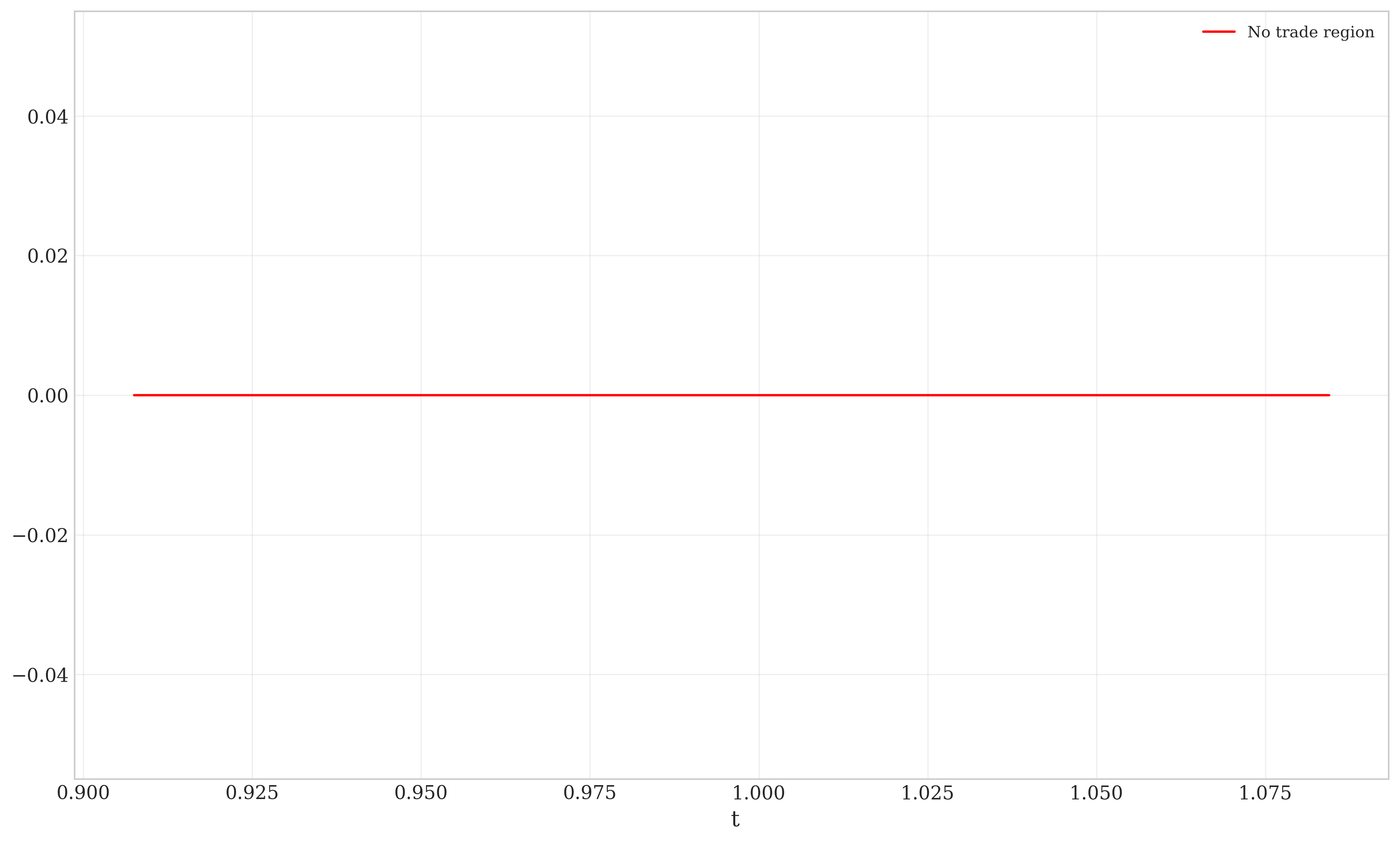}
\caption{Example 2: No-trade region $(q_{i}=0.01,\ i\neq 4;\ q_{4}=9.4)$.}
\label{fig:Ex_2_Notrade_region_q4_94}
\end{figure} 

\begin{table}[htbp]
\centering
\caption{Example 2: Local market data.}
\label{tab:Ex2}
\begin{tabular}{llll}
\hline
CFMM & Trading function $\varphi_{i}$ & Discount rate $\gamma_{i}$ & Local reserves $R^{i}$ \\ 
\hline
$I$   & $\varphi_{1}(R) = \left(R_{1}^{w_{1}}R_{2}^{w_{2}}R_{3}^{w_{3}}\right)^{1/(w_{1}+w_{2}+w_{3})}$ & 0.99 & $(3,0.2,1)$ \\
      & \textit{(weighted geometric mean)}, $w=(3,2,1)$ & & \\[0.5ex]
$II$  & $\varphi_{2}(R) = \sqrt{R_{1}R_{2}}$ & 0.99 & $(10,1)$ \\
      & \textit{(geometric mean)} & & \\[0.5ex]
$III$ & $\varphi_{3}(R) = \sqrt{R_{1}R_{2}}$ & 0.99 & $(1,10)$ \\
      & \textit{(geometric mean)} & & \\[0.5ex]
$IV$  & $\varphi_{4}(R) = \sqrt{R_{1}R_{2}}$ & 0.99 & $(20,50)$ \\
      & \textit{(geometric mean)} & & \\[0.5ex]
$V$   & $\varphi_{5}(R) = R_{1} + R_{2}$ & 0.99 & $(10,10)$ \\
      & \textit{(arithmetic mean)} & & \\
\hline
\end{tabular}
\end{table}

\begin{figure}[htbp]
    \centering
    \begin{subfigure}[b]{0.48\textwidth}
        \centering
        \includegraphics[width=\textwidth,height=0.66\linewidth]{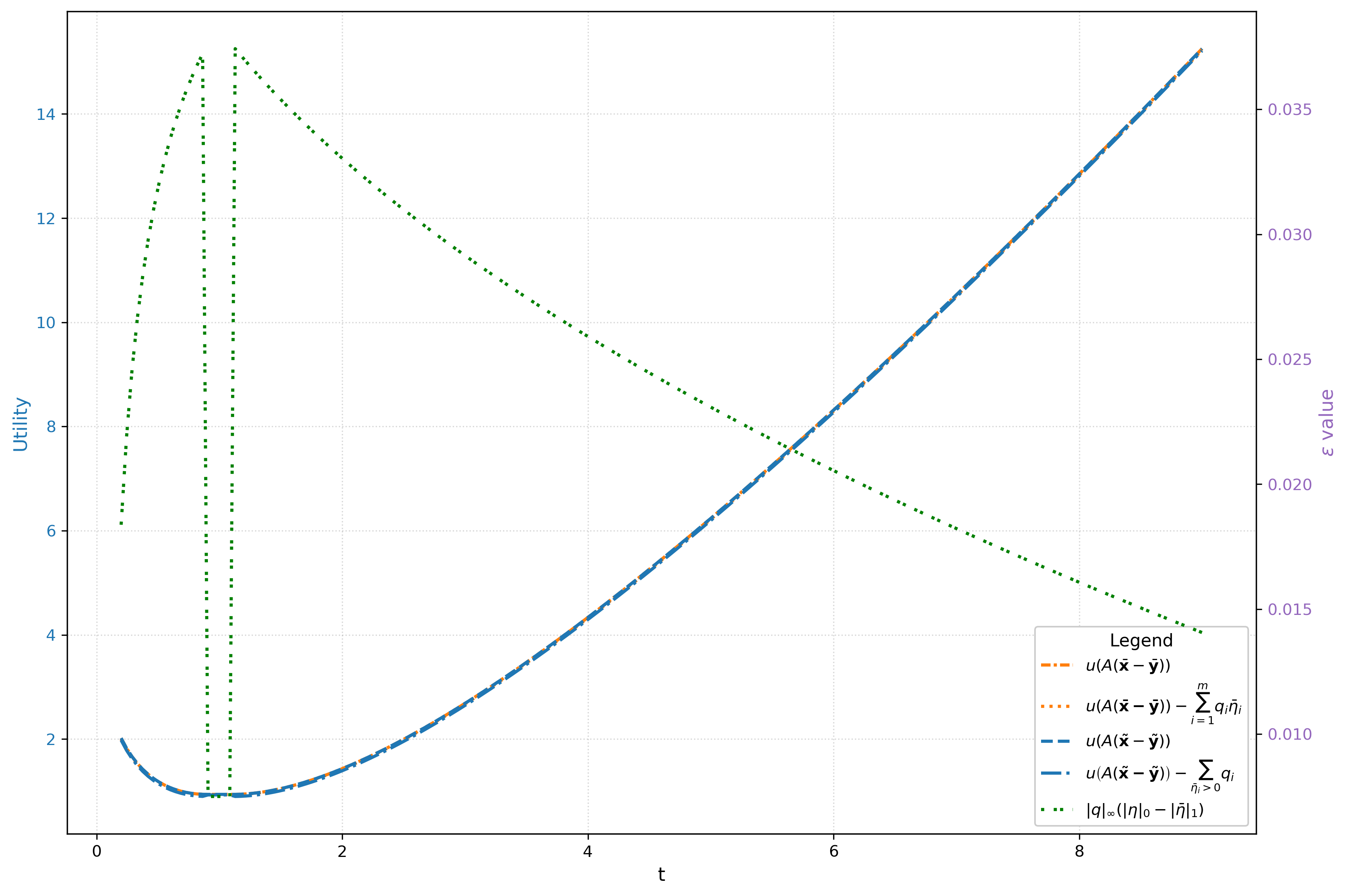}
        \caption{$q_i = 0.01$}
    \end{subfigure}
    \hfill
    \begin{subfigure}[b]{0.48\textwidth}
        \centering
        \includegraphics[width=\textwidth,height=0.66\linewidth]{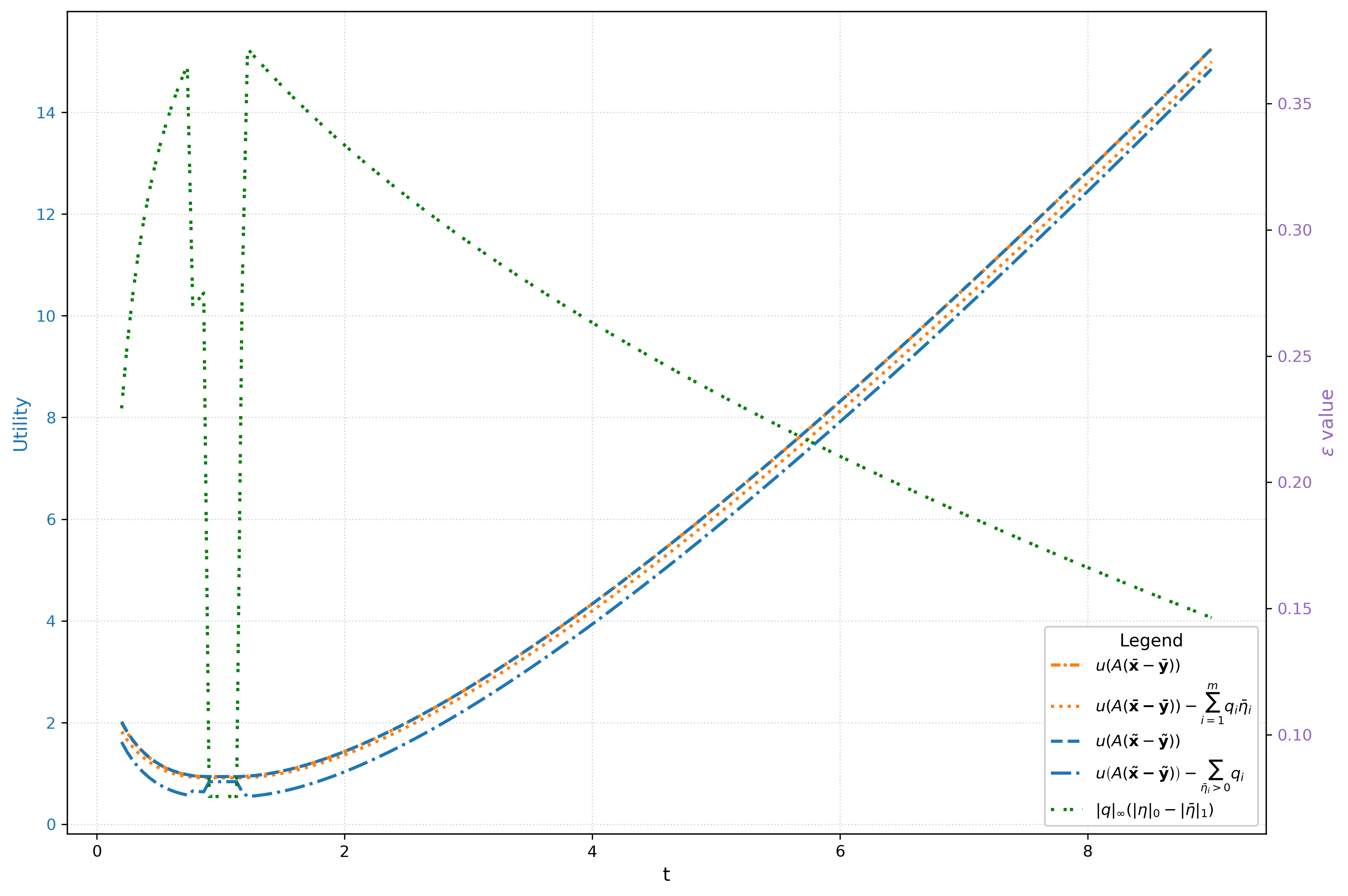}
        \caption{$q_i = 0.1$}
    \end{subfigure}
    
    \vspace{0.5cm} 
    
    \begin{subfigure}[b]{0.48\textwidth}
        \centering
        \includegraphics[width=\textwidth,height=0.66\linewidth ]{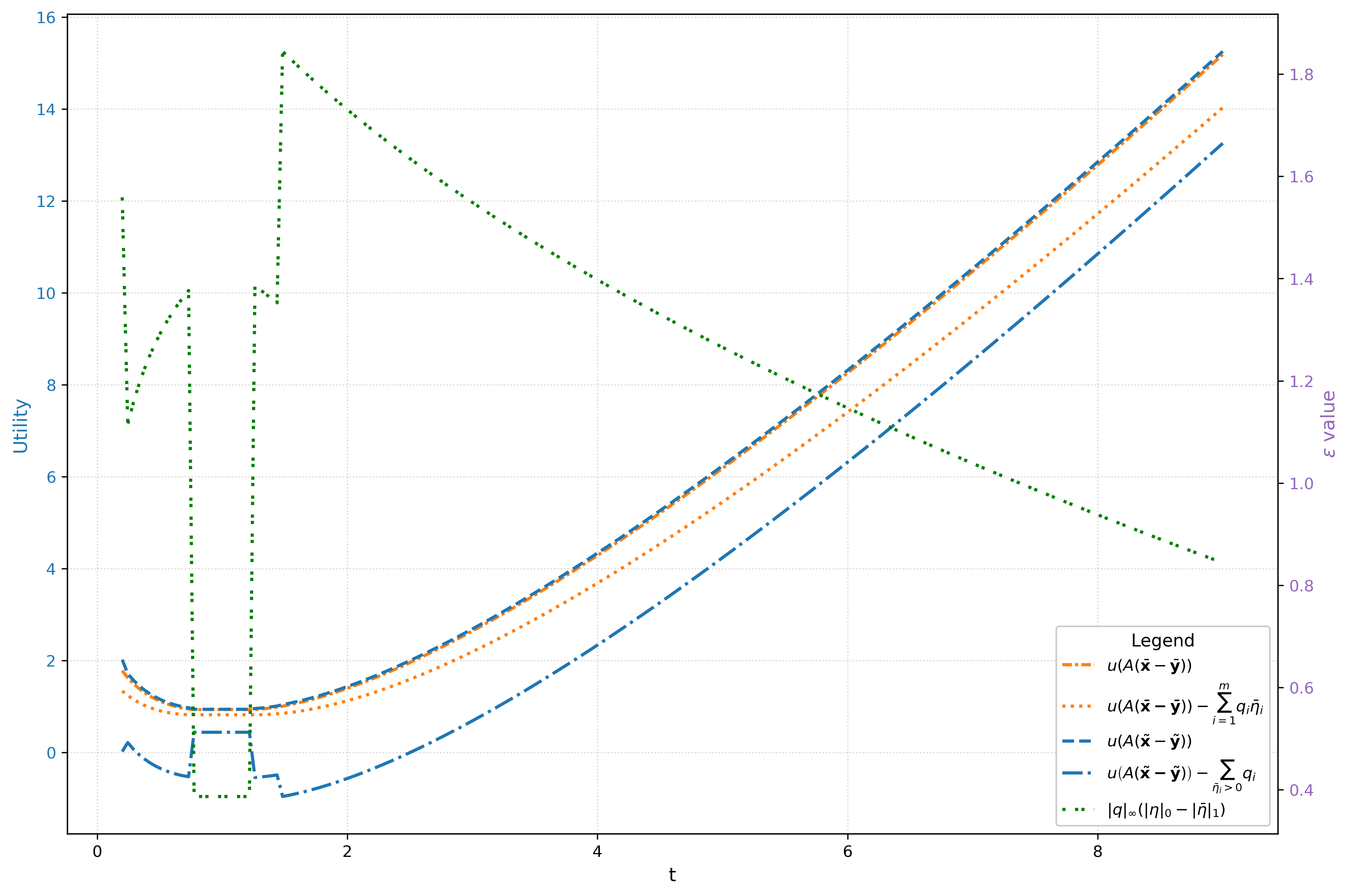}
        \caption{$q_i = 0.5$}
    \end{subfigure}
    \hfill
    \begin{subfigure}[b]{0.48\textwidth}
        \centering
        \includegraphics[width=\textwidth,height=0.66\linewidth]{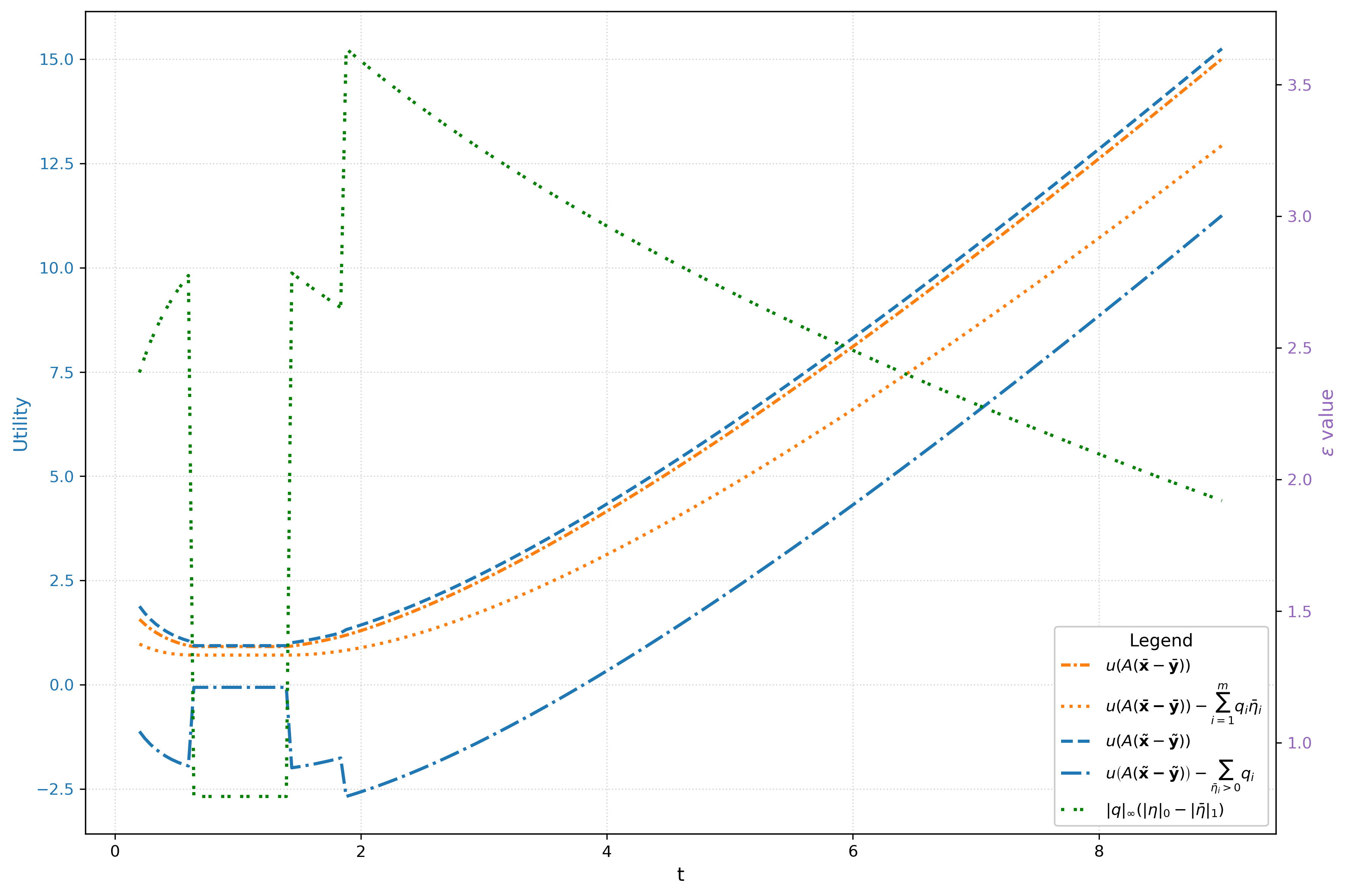}
        \caption{$q_i = 1$}
    \end{subfigure}
    
    \caption{Utility plots for different gas fee $q_i$ values.}
    \label{fig:Ex2_utility_combined}
\end{figure}

\begin{figure}[htbp]
    \centering
    \begin{subfigure}[b]{0.48\textwidth}
        \centering
        \includegraphics[width=\textwidth,height=1.7\linewidth]{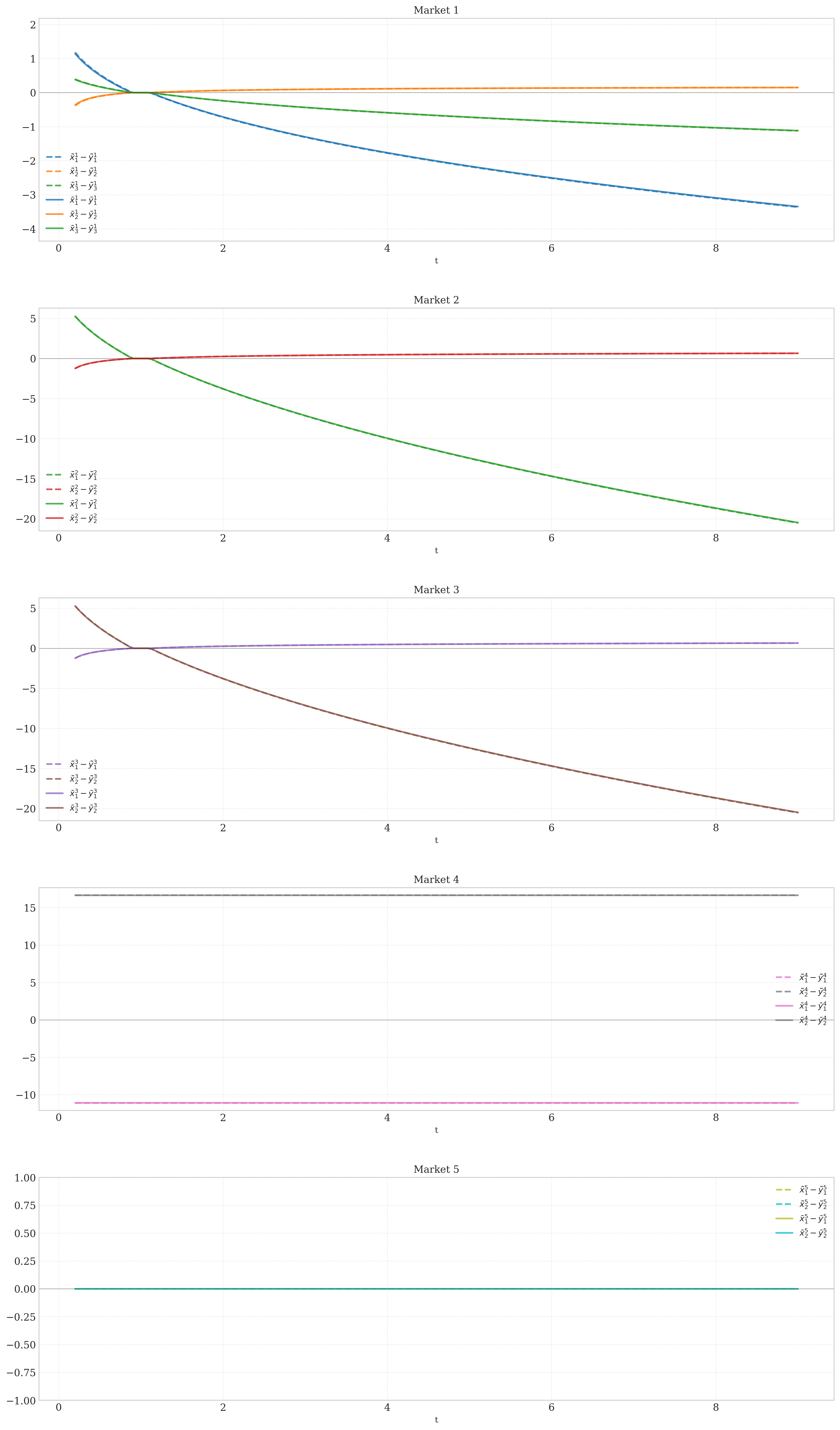}
        \caption{$q_i = 0.01$}
    \end{subfigure}
    \hfill
    \begin{subfigure}[b]{0.48\textwidth}
        \centering
        \includegraphics[width=\textwidth,height=1.7\linewidth]{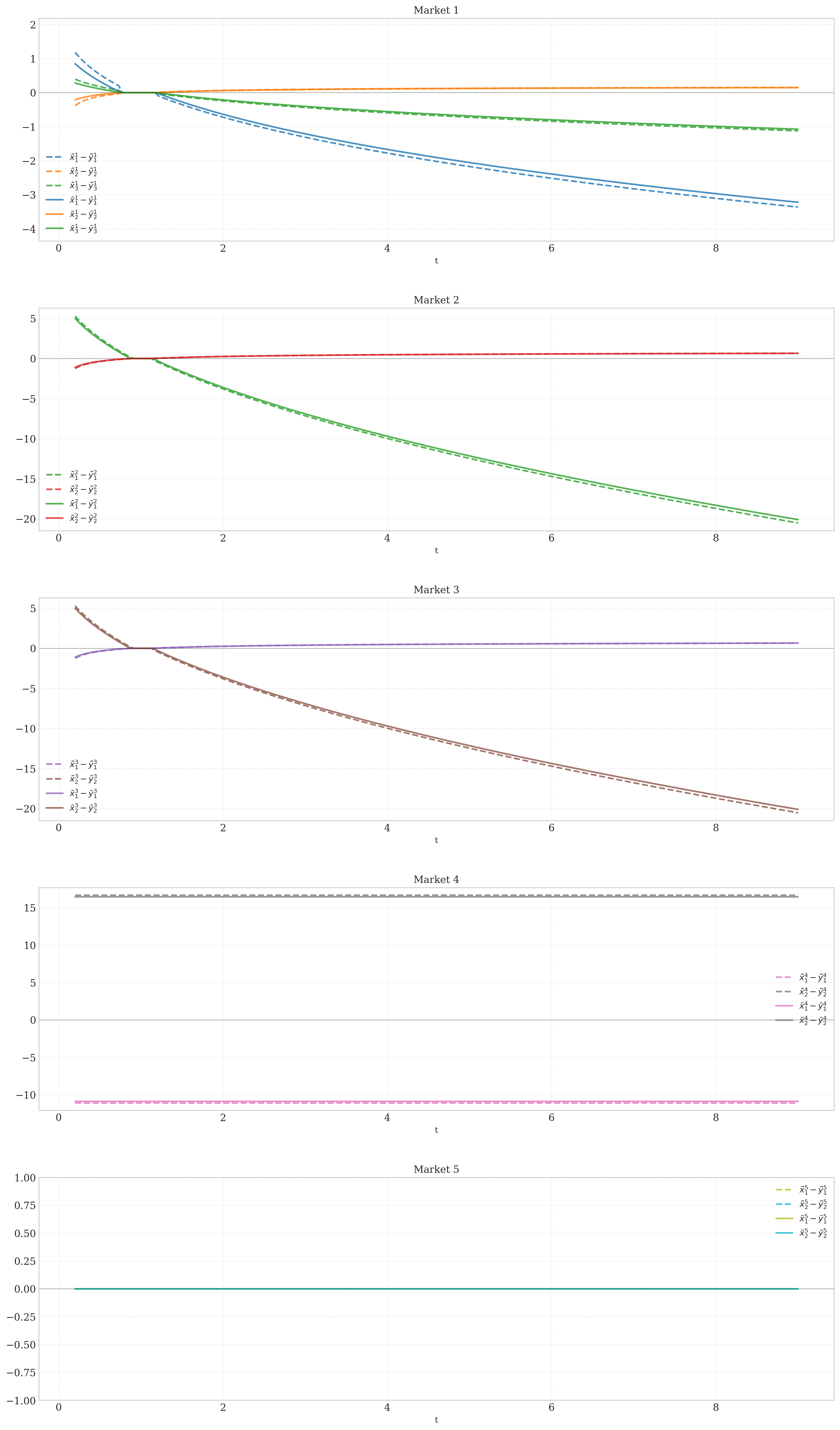}
        \caption{$q_i = 0.1$}
    \end{subfigure}
    
    \vspace{0.5cm} 
    
    \begin{subfigure}[b]{0.48\textwidth}
        \centering
        \includegraphics[width=\textwidth,height=1.7\linewidth]{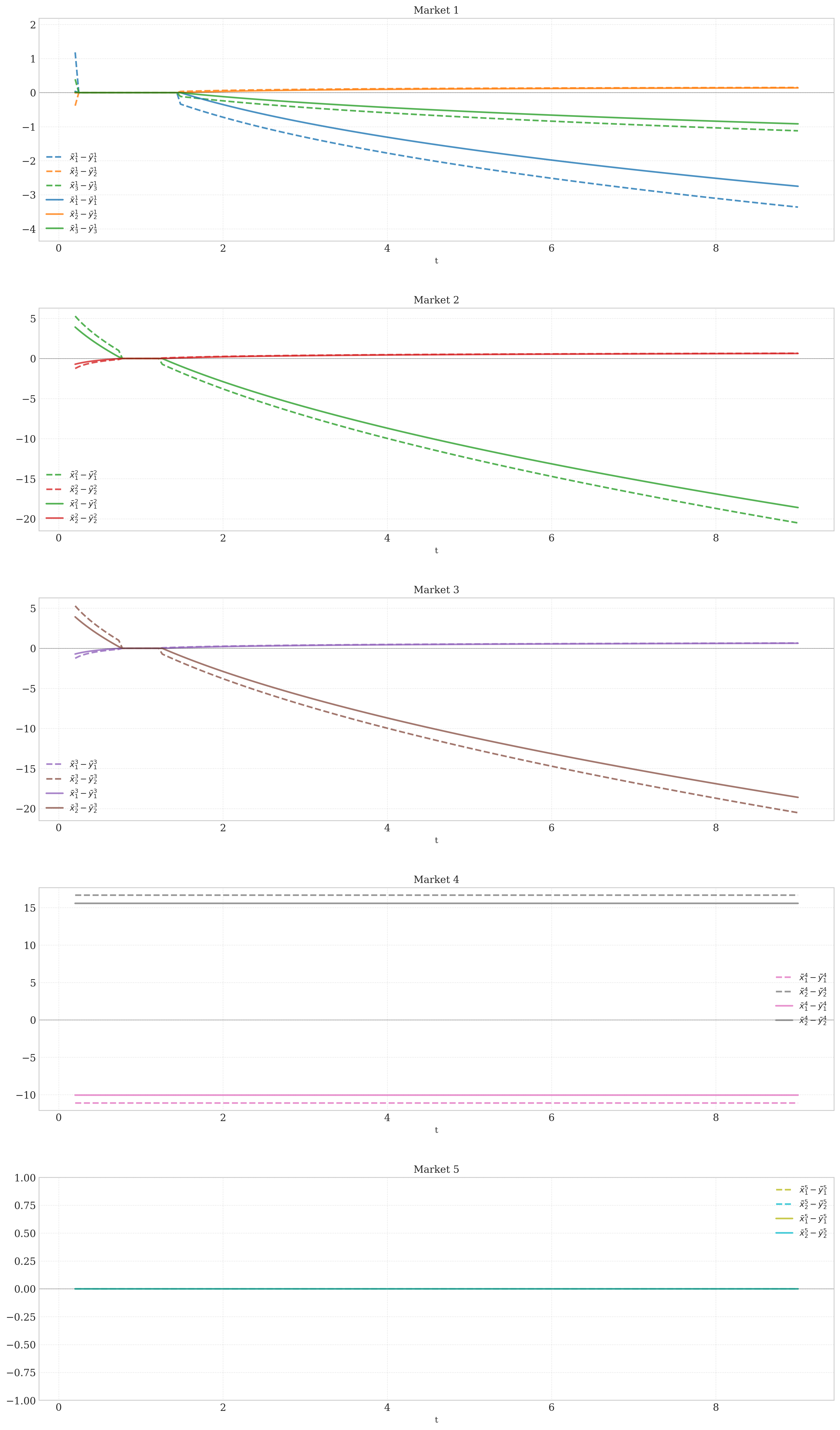}
        \caption{$q_i = 0.5$}
    \end{subfigure}
    \hfill
    \begin{subfigure}[b]{0.48\textwidth}
        \centering
        \includegraphics[width=\textwidth,height=1.7\linewidth]{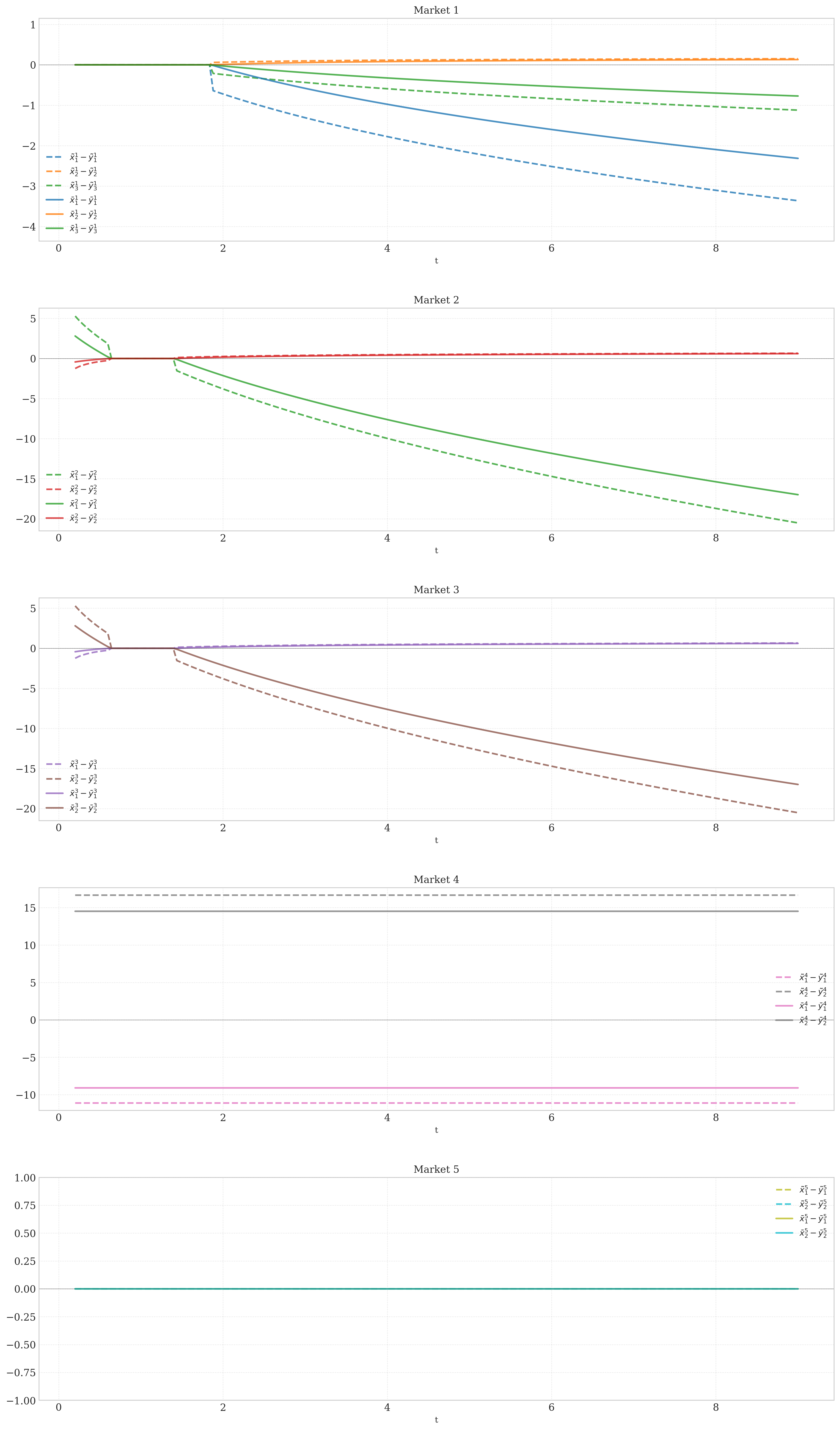}
        \caption{$q_i = 1$}
    \end{subfigure}
    
    \caption{Trade comparisons for different gas fee $q_i$ values.}
    \label{fig:Ex2_trades_comparison}
\end{figure}

\section{Conclusions}

We have studied the optimal routing problem across multiple heterogeneous CFMMs in the presence of fixed execution costs modeling gas fees. The resulting optimization problem departs from classical
convex formulations by combining nonconvex trading constraints with market
activation decisions, which are introduced to enforce the effect of those fees.

Our first contribution has been the derivation of necessary optimality conditions for a relaxation approximation of the routing problem under mild regularity assumptions.
By establishing a KKT system under an appropriate KRZ constraint qualification, we obtained an explicit
characterization of optimal trades linking marginal utilities, pool prices,
fees, and activation variables. This analysis highlights the role of gas fees in
expanding the no-trade region and in shaping optimal routing decisions across
fragmented liquidity pools.

We have further shown that global optimality can be characterized without
imposing global convexity assumptions on the invariant functions. Using tools
from generalized convexity, we established sufficient optimality conditions
under pseudoconcavity of the utilities and quasilinearity of the trade functions. This result extends existing CFMM routing theory beyond convex
models, allowing for a broader class of admissible invariant functions while
preserving optimality.

A rigorous comparison between the relaxed routing problem and the original
mixed-integer formulation with fixed activation costs was also provided.
Explicit approximation bounds were derived, quantifying the utility gap induced
by relaxation and clarifying the conditions under which relaxed solutions offer
meaningful approximations to the original problem. This analysis also provides a
theoretical justification for relaxation-based approaches potentially usable in related questions.

From a financial perspective, our results offer a mathematical interpretation of
how fixed execution costs suppress arbitrage opportunities in decentralized exchanges. The framework developed in this paper thus contributes to a deeper understanding of no-trade conditions and execution frictions in DeFi markets.

All in all, our results contribute to the growing literature
that bridges optimization theory and decentralized market design, showing how
execution frictions fundamentally reshape optimal trading behavior even in
deterministic, on-chain environments.
Several directions for future research naturally arise from them. Extensions to stochastic utility formulations or dynamic routing under uncertainty would further enrich the model. Additionally, the integration of alternative execution constraints or protocol-specific
features may provide further insights into optimal trading and market design in decentralized financial systems.

\bigskip

\noindent \textbf{Acknowledgements:}
Part of this work was carried out when the authors were visiting the Vietnam Institute for Advanced Study in Mathematics (VIASM) throughout March and April 2025; the authors would like to thank VIASM for its hospitality during this period. This research has been partially supported by ANID--Chile under project Fondecyt Regular 1241040 (Lara), by MICIU / AEI / 10.13039 / 501100011033 / and by ERDF, EU (Projects PID2024-156273NA-I00 (Sama) and PID2024-158823NB-I00 (Escudero)), by MICIU / AEI / 10.13039 / 501100011033 / European Union NextGenerationEU / PRTR (Project CPP2024-011557 (Escudero and Sama)), and by COST (European Cooperation in Science and Technology) through COST Action Stochastica CA24104 (Escudero).

\end{document}